%% file: Main.tex
\documentclass[preprint,11pt]{elsarticle}

\usepackage{graphicx}
\usepackage[utf8]{inputenc}
\usepackage{multirow}
\usepackage{comment}
\usepackage{mdwlist}
\usepackage{amsfonts}
\usepackage{ulem}

\usepackage{amsmath}
\usepackage{amsthm}
\usepackage{amssymb}
\usepackage{fancyhdr}

\usepackage{titlesec}
\usepackage{indentfirst}
\usepackage{booktabs}
\usepackage{color}
\usepackage{latexsym}
\usepackage{amscd}
\usepackage{esvect}
\usepackage{upgreek}
\usepackage{caption}
\usepackage[a4paper]{geometry}
\usepackage{mathrsfs}
\usepackage[numbers]{natbib}
\usepackage{float}

\newcommand{\e}{\varepsilon}
\newcommand{\bx}{\boldsymbol{x}}

\newcommand{\bv}{\boldsymbol{v}}
\newcommand{\blue}{\textcolor{black}}
\newcommand{\red}{\textcolor{black}}

\newtheorem{lemma}{Lemma}
\newtheorem{theorem}{Theorem}

\begin{document}
\title{A Bi-fidelity based asymptotic-preserving neural network for the semiconductor Boltzmann equation and its inverse problem}

\author[label1]{Liu Liu}
\author[label2]{Xueyu Zhu}
\author[label3]{Zhenyi Zhu}

\affiliation[label1]{organization={Department of Mathematics, The Chinese University of Hong Kong},
            city={Hong Kong},
            country={China}}

\affiliation[label2]{organization={Department of Mathematics, The University of Iowa},
            city={Iowa City},
            postcode={100084},
            country={USA}}
            
\affiliation[label3]{organization={Department of Mathematics, The Chinese University of Hong Kong},
            city={Hong Kong},
            country={China}}

\begin{abstract}
 This paper introduces a Bi-fidelity Asymptotic-Preserving Neural Network (BI-APNNs) framework, designed to efficiently solve forward and inverse problems for the semiconductor Boltzmann equation. Our approach builds upon the Asymptotic-Preserving Neural Network (APNNs) methodology \cite{APNN-transport}, which employs a micro-macro decomposition to handle the model's multiscale nature. We specifically address a key bottleneck in the original APNNs: the slow convergence of the macroscopic density $\rho$ in the near fluid-dynamic regime, i.e., for small Knudsen numbers $\varepsilon$. The core innovation of BI-APNNs is a novel bi-fidelity decomposition of the macroscopic quantity $\rho$, which accurately approximates the true density at small $\varepsilon$, and can be efficiently pre-trained. A separate and more compact neural network is then tasked with learning only the minor correction term, $\rho_{\text{corr}}$. This strategy not only significantly {\it accelerates} the training convergence but also improves the accuracy of the forward problem solution, particularly in the challenging fluid-dynamic limit. Meanwhile, we demonstrate through extensive numerical experiments that our new BI-APNNs yields substantially more accurate and robust results for inverse problems compared to the standard APNNs. Validated on both the semiconductor Boltzmann and the Boltzmann-Poisson systems, our work shows that the bi-fidelity formulation is a powerful enhancement for tackling multiscale kinetic equations, especially when dealing with inverse problems constrained by partial observation data.\end{abstract}


\maketitle

\section{Introduction} Kinetic equations form the basis for modeling a wide range of physical systems where non-equilibrium dynamics is dominant, spanning from rarefied gas flows and plasma physics to the behavior of electrons in semiconductor devices \cite{Semi-Book}. As a bridge between the microscopic world of individual particles and the macroscopic world of continuum mechanics, these equations, particularly the Boltzmann-type models, are indispensable \cite{A15}. These equations are essential for design, optimization, control, and inverse problems, including semiconductor device design, gas flow channel topology optimization, and risk management in quantitative finance \cite{CS21}. Numerous applications necessitate the identification of unknown or optimal parameters in Boltzmann-type equations or mean-field models \cite{AHP15,Caflisch, CFP13, Cheng11, Lai}.

Solving such kinetic models by using traditional numerical methods is challenging. The challenges are various, including the high dimensionality of the phase space and design of the method to capture dynamics across disparate spatial and temporal scales. 
It is more difficult to solve real-world problems and their parameter estimation, especially when there is only sparse data available. Data-driven approaches, including machine learning methods have gained significant popularity in recent decades \cite{EW21}. Among these, Physics-Informed Neural Networks (PINNs) have emerged as a powerful framework. By embedding the governing partial differential equations (PDEs) directly into the neural network's loss function, PINNs can effectively approximate solutions even with minimal data, and have been successfully applied to a variety of forward and inverse problems \cite{chen2020physics, lou2021physics,  mao2020physics, zhu2025physicssolver}.

Nevertheless, the intrinsic multiscale character of kinetic equations presents a subtle but significant hurdle for standard PINNs frameworks. The system's behavior is often dictated by a small, dimensionless parameter---the Knudsen number, $\varepsilon$---which can lead to severe numerical stiffness. A naive application of PINNs often fails to preserve the correct physical behavior in the asymptotic limit as $\varepsilon \to 0$, resulting in a loss of accuracy and robustness precisely in the regimes of most physical interest \cite{Jin-Ma-Wu1}. To address this challenge, the principles of Asymptotic-Preserving (AP) schemes, a well-established tool in classical numerical analysis \cite{AP-Review}, were integrated into the neural network framework, leading to the Asymptotic-Preserving Neural Network (APNN) \cite{Guilia-Zhu, APNN-transport, Jin-Ma-Wu1,liu2025asymptotic}. By employing a micro-macro decomposition, APNNs enforce the correct asymptotic structure at the level of the loss function, ensuring accuracy and efficiency across different physical regimes.

Despite its advantages, the APNNs methodology reveals a practical limitation: in the crucial fluid-dynamic regime where $\varepsilon$ is very small, the neural network's convergence for the macroscopic density, $\rho$, can be prohibitively slow. This paper introduces the so-called \textbf{Bi-fidelity Asymptotic-Preserving Neural Networks (BI-APNNs)} to address these challenges. Our core idea is to re-formulate the learning task by decomposing the density into a dominant, low-fidelity component and a minor, high-fidelity correction: $\rho = \rho_{\text{diff}} + \rho_{\text{corr}}$. 
The $\rho_{\text{diff}}$ term, which represents the dominant component of the density profile in the asymptotic limit, can be efficiently pre-trained or computed. A much smaller and more agile network is then dedicated to learning only the correction term, $\rho_{\text{corr}}$. This bi-fidelity strategy yields a twofold benefit: it substantially accelerates the training process and simultaneously improves the final accuracy of the forward problem solution. Crucially, we find that this enhanced fidelity provides a decisive advantage in solving inverse problems, leading to more reliable and accurate parameter estimations.

This paper focuses specifically on the Boltzmann semiconductor equation, a fundamental model in modern electronics \cite{Ansgar}. We aim to address two fundamental challenges: first, the efficient and accurate simulation of the system's evolution, known as the forward problem; and second, the inverse problem of estimating unknown physical parameters, such as the collision cross-section, given limited observational data of the particle density.


The main contribution of this work is to design a new and novel bi-fidelity based APNNs approach to significantly accelerate the training process in previously developed APNNs, and apply it to solve the semiconductor Boltzmann(-Poisson) equation and its parameter inference using limited data. The rest of the paper is organized as follows. Section \ref{sec:model} introduces the semiconductor Boltzmann model and the micro-macro decomposition. In Section \ref{sec:NN-general}, we present the formulation of our proposed BI-APNN framework, with convergence analysis provided in Section \ref{sec:analysis}. Extensive numerical experiments are shown in Section \ref{sec:Num} to validate the robustness and efficiency of our method in solving both forward problem and parameters inference, followed by conclusion in the final section.

\section{The models and Micro-macro decomposition method}
\label{sec:model}

\noindent\textbf{The semiconductor-Boltzmann.}
\blue{We consider the semiconductor Boltzmann equation \cite{JP2000} given by}
\begin{equation}
\label{Boltz-eqn}
\e \partial_t f + v \cdot \nabla_{x} f +  \nabla_{x} \phi\, \cdot \nabla_{v} f = \frac{1}{\e}\mathcal{Q}(f), 
\end{equation}
where $f(t,x,v)$ is the probability density distribution for particles located at $x \in \mathcal{D}\subset \mathbb{R}^{d_x}$ with velocity $v\in \mathbb{R}^{d_v}$. In our model, $\e$ is the Knudsen number defined as the ratio of the mean free path and the typical length scale. In our application, $\e$ varies from $O(1)$, the kinetic regime, to $\e\ll 1$, the diffusion regime.  the  $\phi(t,x)$ is external electric potential. 

The anisotropic collision operator $\mathcal{Q}$ describes a linear approximation of the electron-phonon interaction, given by 
$$  \mathcal{Q}(f)(t,x,v) = \int_{\mathbb{R}^{d_v}} \sigma(v,w) \left( M(v)f(t,x,w) - M(w)f(t,x,v) \right) dw, $$
with $M$ the normalized Maxwellian 
$ M(v) = \frac{1}{\pi^{d/2}} e^{-|v|^2}$. 
Here \blue{$\sigma(v,w)$} denotes the scattering coefficient for the electron-phonon collisions. The collision frequency is defined as
\blue{$$ \lambda(v) = \int_{\mathbb{R}^{d_v}} \sigma(v,w)M(w) dw. $$}
We refer the readers to \cite{Rode,Ansgar,Sze81} for more physical background.

Define $n_x$ the unit outward normal vector on the spatial boundary 
$\partial\mathcal{D}$. Let $\gamma = \partial\mathcal{D}\times\Omega$, then the phase boundary can be split into an outgoing boundary $\gamma_{+}$, incoming boundary $\gamma_{-}$ and a singular boundary $\gamma_0$, which are defined by
\begin{equation}
\label{gamma}
\begin{aligned}
& \gamma_{+}:= \left\{ (x,v)\in \partial\mathcal{D}\times\Omega: 
v \cdot n_x >0 \right\}, \\[6pt]
& \gamma_{-}:= \left\{ (x,v)\in \partial\mathcal{D}\times\Omega: v\cdot n_x <0 \right\}, \\[6pt]
& \gamma_0 := \left\{ (x,v)\in \partial\mathcal{D}\times\Omega: v\cdot n_x =0\right\}. 
\end{aligned}
\end{equation}
The inflow boundary condition is given by
\begin{equation} f(t,x,v) = f_{\text{BC}}(t,x,v), \qquad \text{for   }\, 
(t,x,v) \in [0,T]\times \gamma_{-}. \end{equation}
We assume the initial condition that
$$ f(t=0,x,v)=f_{\text{IC}}(x,v). $$

\noindent\textbf{The Boltzmann-Poisson system.}
In this work, we also study the {\it nonlinear} Boltzmann-Poisson system, where the electric potential $\phi(t,x)$ is solved by the Poisson equation as below: 
\begin{equation}
\label{eqn:BP}
\left\{
\begin{array}{ll}
\displaystyle\e \partial_t f + v \cdot \nabla_{x} f + 
\nabla_{x} \phi\, \cdot \nabla_{v} f  = \frac{1}{\e}\mathcal{Q}(f), \\[6pt]
\displaystyle\beta \Delta_x\phi = \int_{\mathbb{R}^{d_v}}f dv - c(x), \\[6pt]
\displaystyle\phi(t,0) = 0, \qquad \phi(t,1) = V, 
\end{array}
\right.
\end{equation}
with $\beta$ the scaled Debye length, $V$ the applied bias voltage and $c(x)$ the doping profile.

\subsection{The micro-macro decomposition method}

We employ the micro-macro decomposition technique \cite{MM08} and derive the micro-macro system for the semiconductor Boltzmann equation \eqref{Boltz-eqn}. 
Assume the ansatz 
\begin{equation}\label{Ans} f = \Pi f + \e g, 
\end{equation}
where $g: = g(t, x, v)$ and the notation $\Pi f :=  \langle f \rangle\, M(v)$ is defined by  
\begin{equation*}\label{bracket}
    \langle f \rangle := \int_{\mathbb{R}} f(t,x,v)\,dv = \rho(t, x).
\end{equation*}

Inserting the ansatz \eqref{Ans} into  \eqref{Boltz-eqn}, one gets
\blue{
\begin{equation}\label{PP}
\e \partial_t (\Pi f) + \e^2 \partial_t g + v \cdot \nabla_x (\Pi f) + \e \nabla_x \cdot (v g)  - 2 v \rho \cdot \nabla_x \phi \, M(v) + \e \nabla_x\phi\ \cdot \nabla_v g  = \mathcal{Q}(g), 
\end{equation}
where $\nabla_v (\Pi f) = \rho\,\nabla_v M(v) = - 2 \rho\, v M(v)$ is used. }

Take the projection operator $\Pi$ on both sides of \eqref{PP}, then 
\blue{
$$ \e \partial_t (\Pi f) + \e^2 \partial_t (\Pi g) + \nabla_x \rho \left(\int v M(v) dv\right) \cdot M(v) + \e \Pi \nabla_x \cdot (v g) + \e \nabla_x\phi\, \cdot \Pi(\nabla_v g) = 0, $$}
which is equivalent to
\blue{\begin{equation}\label{Macro}
\partial_t (\Pi f) + \Pi \nabla_x \cdot (v g) + \nabla_x \phi \, \cdot \Pi(\nabla_v g) = 0, 
\end{equation}}
where $\Pi g = 0$ and $\int_{\mathbb R} v M(v) dv = 0$ is used. 
Now subtract \eqref{PP} by \eqref{Macro}, then
\blue{$$
\e^2 \partial_t g  + v \cdot \nabla_x (\Pi f)  + \e (\mathbb{I} - \Pi) \nabla_x \cdot (v g) + \e \nabla_x \phi\, \cdot (\mathbb{I} - \Pi) \nabla_v g - 2 v \rho\, \cdot \nabla_x \phi \, M(v) = \mathcal{Q}(g), 
$$}
that is
\blue{\begin{equation}\label{Micro}
\e^2 \partial_t g +  v \cdot \nabla_x \rho\, M(v)  +  \e (\mathbb{I} - \Pi) \left( v \cdot \nabla_x g + \nabla_x\phi \cdot \nabla_v g \right) - 2 v \rho \cdot \nabla_x \phi \, M(v)  = \mathcal{Q}(g).  
\end{equation}}

As $\e\to 0$, equation \eqref{Macro} stays unchanged, and the microscopic equation \eqref{Micro} gives 
\blue{$$ g = \mathcal{Q}^{-1}\left( v \cdot \nabla_x \rho\, M(v) - 2 v \rho \cdot \nabla_x\phi\, M(v) \right) = \left( \nabla_x\rho - 2 \rho \nabla_x\phi \right) \mathcal{Q}^{-1}(v M(v)). $$} 
Plug $g$ into \eqref{Macro} and integrate over $v$, we derive the drift-diffusion limit \cite{PP91}: 
\blue{\begin{equation}\label{Diffusion} \partial_t \rho = \nabla_x \cdot \left( T \nabla_x\rho - 2 (\rho\, \nabla_x\phi) \right) , \end{equation}}
where the diﬀusion matrix $T$ is defined by 
$T=\int_{\mathbb{R}^{d_v}}\frac{v \bigotimes v M(v)}{\lambda(v)} dv$. 

To summarize, the coupled equations \eqref{Macro}--\eqref{Micro} for $\rho$ and $g$ solve the following system in our micro-macro decomposition framework: 
\blue{
\begin{equation}
\label{MM}
\left\{
\begin{array}{ll}
\displaystyle\partial_t \rho + \nabla_x \cdot \langle v g \rangle + \nabla_x\phi \cdot \langle \partial_v g \rangle = 0, \\[6pt]
\e^2 \partial_t g +  v \cdot \nabla_x \rho\, M(v)  +  \e (\mathbb{I} - \Pi) \left( v \cdot \nabla_x g + \nabla_x\phi \cdot \nabla_v g \right) - 2 v \rho \cdot \nabla_x \phi \, M(v)  = \mathcal{Q}(g), \\[6pt]
\Pi g = 0. 
\end{array} 
\right.
\end{equation}
}
where the third equation is for the mass conservation for $g$, which is crucial to the training of neural networks \cite{APNN-transport} and will be mentioned in section \ref{sec:APNN}. 

As $\e\to 0$, according to \eqref{MM}, the limiting system for $\rho$ and $g$ satisfies
\blue{
\begin{equation}
\label{limit-diff}
\left\{
\begin{array}{ll}
\displaystyle\partial_t \rho + \nabla_x \cdot \langle v g \rangle + \nabla_x\phi \cdot \langle \nabla_v g \rangle = 0, \\[6pt]
\displaystyle v \cdot \nabla_x \rho \, M(v) - 2 v \rho \cdot \nabla_x\phi \, M(v)  = \mathcal{Q}(g) , \\[6pt]
\displaystyle \Pi g = 0, 
\end{array} 
\right.
\end{equation}
}
which is {\it equivalent} to the diffusion limit  given in \eqref{Diffusion}.

\section{Methodology}
\label{sec:NN-general}
This section details the numerical methodology developed for efficiently solving the Boltzmann semiconductor system introduced in Section~\ref{sec:model}. As previously discussed, this multiscale kinetic equation presents significant computational challenges for standard numerical and machine learning methods. To address these challenges, we propose a novel bi-fidelity framework built upon recent advances in physics-informed neural networks.

To logically construct our approach, this section is organized as follows. We begin by reviewing the foundational concepts of Physics-Informed Neural Networks (PINNs). We then discuss their extension, the Asymptotic-Preserving Neural Networks (APNNs), which are specifically designed for multiscale problems. Finally, we introduce our primary contribution: the Bi-fidelity Asymptotic-Preserving Neural Networks (BI-APNNs), and detail their formulation.

\subsection{Physics Informed Neural Networks}
\label{sec:NN}
The fundamental idea of Physics-Informed Neural Networks (PINNs) \cite{raissi2019physics} is to approximate the solution of a partial differential equation (PDE) using a neural network. The network is trained by minimizing a loss function that directly incorporates the underlying physics. This loss function is constructed by summing the mean squared errors of several components: (1) the PDE residual evaluated on a set of collocation points inside the domain, (2) the mismatch between the network's prediction and the prescribed initial conditions, and (3) the mismatch at the boundary conditions. To find the optimal values for the network parameters $\theta$ that are composed of all the weights $w_{ji}$ and bias $b_j$, the neural network is trained by minimizing the following loss function 
\begin{equation}
\label{Loss-PINN}
\begin{aligned}
\mathcal{R}_{\mathrm{PINN}}^{\e} = Loss_{\text{GE}} + \lambda_1 Loss_{\text{BC}}
 + \lambda_2 Loss_{\text{IC}}, 
 \end{aligned}
\end{equation}
where the loss for the residual of the governing PDE is given by: 
\begin{equation}
Loss_{\text{GE}}:=\int_{\mathcal{T}} \int_{\mathcal{D}} \int_{\Omega} \Big| \e \partial_t f_\theta^{\mathrm{NN}}
 +   v \partial_x f_\theta^{\mathrm{NN}} +  \partial_x \phi\, \partial_v f_\theta^{\mathrm{NN}} - \frac{1}{\e}\mathcal{Q}(f_\theta^{\mathrm{NN}})|^2 \, d\bv d\bx dt. 
\end{equation}
The loss for the boundary condition is
\begin{equation} Loss_{\text{BC}}:= \int_{\mathcal{T}}\int_{\gamma_{-}}
\left|\mathcal{B}  \left(\rho_\theta^{\mathrm{NN}} M(v) + \e g_\theta^{\mathrm{NN}}\right) - f_{\text{BC}}\right|^2 \, ds dt,  \end{equation}
and the loss for the initial condition: 
\begin{equation} Loss_{\text{IC}}:=\int_{\mathcal{D}} \int_{\Omega}\left|\mathcal{I}\left(\rho_\theta^{\mathrm{NN}} M(v) +\e g_\theta^{\mathrm{NN}}\right)- f_{\text{IC}} \right|^2
 \, d\bv d\bx, 
\end{equation}
where the initial and boundary conditions are incorporated into the loss function as a regularization or penalty term, 
with penalty parameters $\lambda_1, \lambda_2$ chosen for optimal performance. The most popular method to minimize the loss function over the parameter space are stochastic gradient descent and advanced optimizers such as Adam \cite{Adam14}. After the training process and achieving the optimal set of parameter values $\theta^{\star}$ by minimizing the PINNs loss
\eqref{Loss-PINN}, i.e., 
$$ \theta^{\star} = \text{argmin} \ \mathcal{R}_{\mathrm{PINNs}}^{\e}(\theta), $$
the neural network surrogate $f^{NN}_{\theta^{\star}}(t,x,v)$ can be evaluated at any given point in the temporal and phase space to obtain the solution of \eqref{Boltz-eqn}.

\subsection{The APNNs framework}
\label{sec:APNN}

A key challenge arises from the multiscale nature of the model, characterized by the Knudsen number $\e$. Standard PINNs formulations often fail to respect the asymptotic properties of the kinetic equation, preventing them from capturing the correct macroscopic behavior uniformly across all ranges of $\e$. This can lead to significant inaccuracies, particularly in the fluid-dynamic limit. To overcome this limitation, Asymptotic-Preserving Neural Networks (APNNs) were introduced by Jin et al. \cite{Jin-Ma-Wu2}. The core strategy of the APNNs framework is to design a loss function that explicitly incorporates the model's asymptotic behavior. By doing so, the network can accurately capture the macroscopic dynamics even when the scaling parameter $\e$ is small, thereby possessing the crucial Asymptotic-Preserving (AP) property. The fundamental concept of the APNNs is illustrated in Figure~\ref{fig:APNN}.

\begin{figure}[htbp]
\centering
\includegraphics[width=0.65\textwidth]{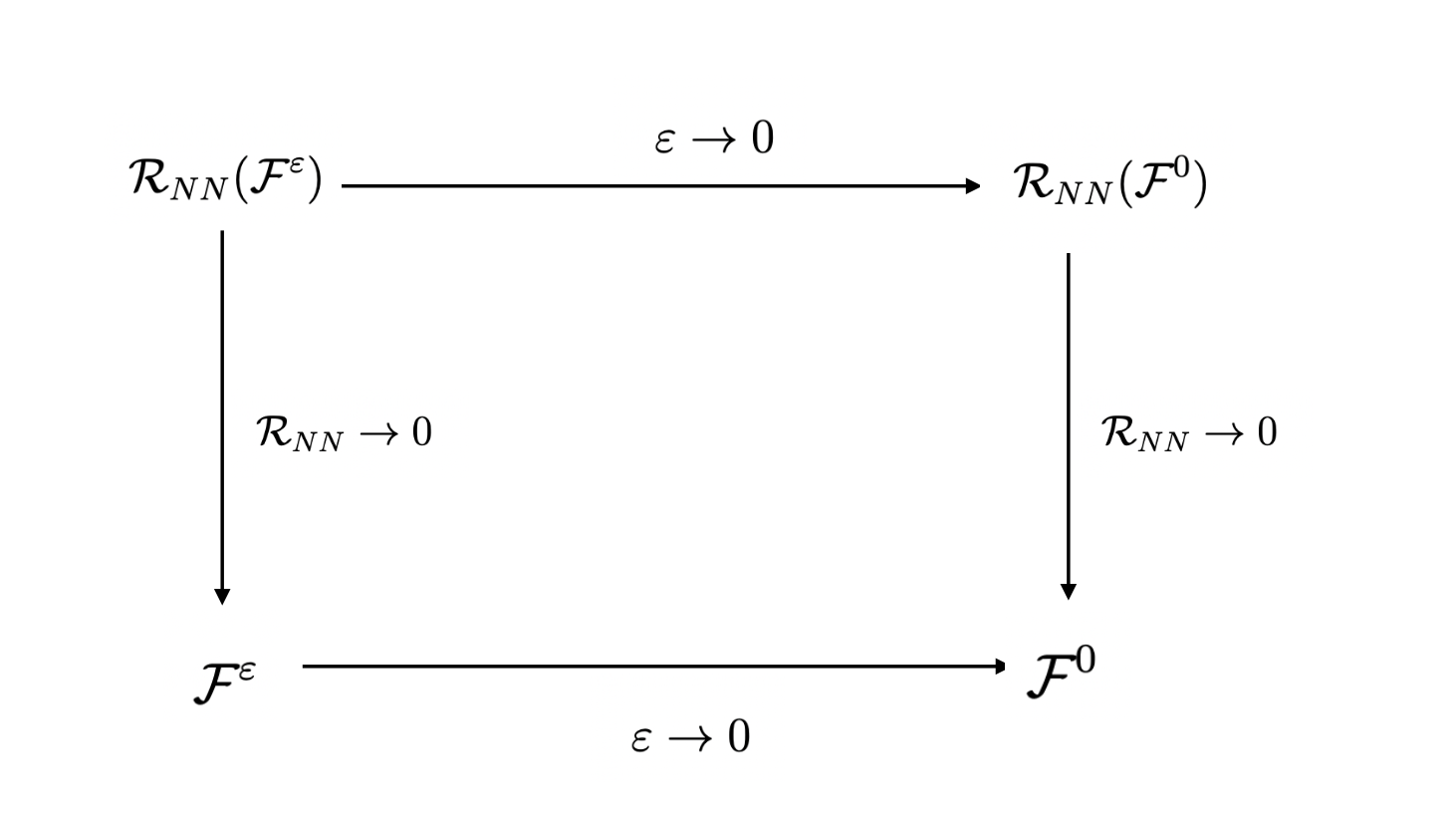}
\caption{Illustration of APNNs. }
\label{fig:APNN}
\end{figure}

Let $\mathcal{F}^{\e}$ be a multiscale model that depends on the scaling parameter $\e$. As $\e\to 0$, it converges to a reduced order limit 
$\mathcal{F}^0$. The solution of $\mathcal{F}^{\e}$ is approximated by the neural network through the imposition of residual term $\mathcal{R}_{NN}(\mathcal{F}^{\e})$, 
whose asymptotic limit is denoted by $\mathcal{R}_{NN}(\mathcal{F}^0)$ as $\e\to 0$. The neural network is called AP if $\mathcal{R}_{NN}(\mathcal{F}^0)$ is consistent
with the residual for the system $\mathcal{F}^0$. 

In APNNs, \cite{Jin-Ma-Wu1,liu2025asymptotic} employ two deep neural networks to parameterize the two functions $\rho(t,\bx)$ and $g(t, \bx, \bv)$, where the notations here are extended to the multi-dimensional spatial and velocity space. Since the density $\rho$ is non-negative, we put an exponential function at the last output layer of the DNN for $\rho$ and denote
\begin{equation}
\rho_\theta^{\mathrm{NN}}(t, \bx):=\exp \left(-\tilde{\rho}_\theta^{\mathrm{NN}}(t, \bx)\right) \approx \rho(t, \bx)
\end{equation}

\blue{We remark that the velocity discretization for $g(t,\bx,\bv)$, where the Hermite quadrature rule is used since $v\in \mathbb{R}^{d_v}$.} This is standard and the same as traditional AP method studied in \cite{JP2000}. Let $g(t, \bx, \bv) = \psi(t, \bx, \bv)M(\bv)$, with $M(\bv) = \frac{1}{\pi^{d_v/2}} e^{-|\bv|^2}$ and 
\begin{equation}\label{Psi} \psi(t,\bx,\bv) = \sum_{k=0}^{N} \psi_k(t,\bx) H_k(\bv)
\end{equation}
is the Hermite expansion with $N$ being the order. 
We then compute the collision operator $Q$ in \eqref{Boltz-eqn} as follows
\blue{$$ Q(g)(\bv) = M(\bv) \sum_{j=0}^{N_{\bv}} \sigma(\bv, v_j)\, \psi(v_j)\, w_j - \lambda(\bv) g(\bv), $$} 
with \blue{$\lambda(\bv) = \sum_{j=0}^{N_{\bv}} \sigma(\bv, v_j)\, w_j$}.
The derivative in $\bv$ for $\psi$ is given by
\begin{equation}
\begin{aligned}
\partial_{\bv} \psi 
 = \sum_{j=0}^{N_v} \psi(v_j)\, C_j(\bv),
\end{aligned}
\end{equation}
where $C_j(\bv) = \sum_{k=0}^{N} \sqrt{2k}\, H_k(v_j) H_{k-1}(\bv) w_j$. 

In order to obtain approximation for $g$, we let $\tilde{\psi}_\theta^{\mathrm{NN}}(t,\bx,\bv)$ be the output of a fully-connected neural network with input $t$, $\bx$ and $\bv$, then impose a post-processing step to guarantee the mass conservation in \eqref{MM}. Define 
\begin{equation}\label{g_NN}
 g_\theta^{\mathrm{NN}}(t, \bx, \bv) =
 \psi_{\theta}^{\mathrm{NN}}(t,\bx,\bv)M(\bv) := \tilde{\psi}_\theta^{\mathrm{NN}}(t,\bx,\bv)M(\bv)-\Pi\left(\tilde{\psi}_\theta^{\mathrm{NN}}(t,\bx,\bv)M(\bv)\right), 
\end{equation}
as an approximation for $g$, now 
$\Pi g_\theta^{\mathrm{NN}} = 0 $ is automatically satisfied. Moreover, $\nabla_{\bv}g_\theta^{\mathrm{NN}}$ is computed by 
\begin{equation}
    \nabla_{\bv}g_\theta^{\mathrm{NN}} = \sum_{j=0}^{N_{\bv}}\psi_\theta^{\mathrm{NN}}(t,\bx,v_j) C_j(\bv) M(\bv)-2\bv M(\bv)\psi_\theta^{\mathrm{NN}}(t,\bx,\bv) + 2\bv M(\bv) \langle \psi_\theta^{\mathrm{NN}}(t,\bx,\bv) M(\bv) \rangle.
\end{equation}
where $\nabla_{\bv}\psi_{\theta}^{NN} = \sum_{j=0}^{N_{\bv}} \psi_\theta^{\mathrm{NN}}(t,\bx,v_j)C_j(\bv)$.  

\textbf{APNNs Loss.} For the APNNs method, the physics-informed loss function is constructed from the residuals of the coupled macro-micro system \eqref{Macro}--\eqref{Micro}:
\begin{equation}
\label{Loss-APNN}
\begin{aligned}
\mathcal{R}_{\mathrm{APNNs}}^{\e} = & \frac{1}{|\mathcal{T} \times \mathcal{D}|} \int_{\mathcal{T}} \int_{\mathcal{D}}\left|\partial_t \rho_\theta^{\mathrm{NN}}+\nabla_{\bx} \cdot\left\langle\bv g_\theta^{\mathrm{NN}}\right\rangle + \red{ \nabla_{\bx}\phi \cdot  \left\langle \nabla_{\bv} g_\theta^{\mathrm{NN}}\right\rangle } \right|^2 \, d\bx d t \\[6pt]
& + \frac{1}{|\mathcal{T} \times \mathcal{D} \times \Omega|} \int_{\mathcal{T}} \int_{\mathcal{D}} \int_{\Omega} \Big| \e^2 \partial_t g_\theta^{\mathrm{NN}}
 + \e(I-\Pi) \left(\bv \cdot \nabla_{\bx} g_\theta^{\mathrm{NN}} + \red{\nabla_{\bx}\phi \cdot \nabla_{\bv} g_\theta^{\mathrm{NN}} }\right) \\[6pt]
&   - 2\bv \cdot \nabla_{\bx}\phi\, \rho_\theta^{\mathrm{NN}}\, M(\bv)  
 + \bv \cdot \nabla_{\bx} \rho_\theta^{\mathrm{NN}} M(\bv)  - \mathcal{Q}( g_\theta^{\mathrm{NN}} ) \Big|^2 \,  d\bv d\bx d t  \\[6pt]
& + \frac{\lambda_1}{| \mathcal{T} \times \partial \mathcal{D} \times \Omega |} \int_{\mathcal{T}}\int_{\gamma_{-}}\left|\mathcal{B}  \left(\rho_\theta^{\mathrm{NN}} M(\bv) + \e g_\theta^{\mathrm{NN}}\right) - f_{\text{BC}}\right|^2 \, ds dt\\[6pt]
 &+\frac{\lambda_2}{|\mathcal{D} \times \Omega|} \int_{\mathcal{D}} \int_{\Omega}\left|\mathcal{I}\left(\rho_\theta^{\mathrm{NN}} M(\bv) +\e g_\theta^{\mathrm{NN}}\right)- f_{\text{IC}} \right|^2
 \, d\bv d\bx.
\end{aligned}
\end{equation}

As $\e\to 0$, the above first two terms arisen from the model equation lead to 
\begin{equation*}
\begin{split}
& \mathcal{R}_{\mathrm{APNNs}} \text {, residual}  =  \frac{1}{|\mathcal{T} \times \mathcal{D}|} \int_{\mathcal{T}} \int_{\mathcal{D}}\left|\partial_t \rho_\theta^{\mathrm{NN}}+\nabla_x \cdot\left\langle\bv g_\theta^{\mathrm{NN}}\right\rangle + \red{ \nabla_{\bx}\phi \cdot  \left\langle \nabla_{\bv} g_\theta^{\mathrm{NN}}\right\rangle } \right|^2 \, d \bx d t \\[6pt]
& +\frac{1}{|\mathcal{T} \times \mathcal{D} \times \Omega|} \int_{\mathcal{T}} \int_{\mathcal{D}} \int_{\Omega}\left|   - 2\bv \cdot \nabla_{\bx}\phi\, \rho_\theta^{\mathrm{NN}}\, M(v) 
+ \bv \cdot \nabla_{\bx} \rho_\theta^{\mathrm{NN}}\, M(v)  -  \mathcal{Q}(g_\theta^{\mathrm{NN}})\right|^2 \, d \bv d \bx d t, 
\end{split}
\end{equation*}
which is identical to the loss function associated with the limiting system \eqref{limit-diff}. Consequently, the proposed loss function possesses the Asymptotic-Preserving (AP) property, a key feature that standard PINN formulations lack.

\subsection{Bi-fidelity method and Bi-APNNs Method} 
The bi-fidelity method is a computational approach that leverages solutions from both low-fidelity and high-fidelity models to balance accuracy and computational efficiency. By combining a coarse, computationally inexpensive model with a corrective term derived from a more accurate but resource-intensive model, it aims to approximate high-fidelity solutions at a reduced cost. We consider two ways of decomposition for \(\rho\) based on the bi-fidelity idea, defined as follows: 
\begin{equation}
\begin{aligned}
\label{EI-system}
   \text{(i) Explicit formulation: }\quad \rho_{\text{bi}}(t, \mathbf{x}) & = \rho_{\textit{diff}}(t, \mathbf{x}) + \varepsilon \rho_{\textit{corr}}(t, \mathbf{x}), \\[4pt]
   \text{(ii) Implicit formulation: }\quad \rho_{\text{bi}}(t, \mathbf{x}) & = \rho_{\textit{diff}}(t, \mathbf{x}) + \rho_{\textit{corr}}(t, \mathbf{x}),
\end{aligned}
\end{equation}
where the solution  $\rho_{\textit{diff}}$ from \eqref{Diffusion} is combined with a correction term.  The explicit formulation is motivated by the asymptotic property of the semiconductor Boltzmann equation. When the Knudsen number~$\varepsilon$ tends to zero, the solution gradually approaches the diffusion limit, and the microscopic correction becomes negligible. By expressing the macroscopic variable as~$\rho = \rho_{\textit{diff}} + \varepsilon \rho_{\textit{corr}}$, this scaling explicitly reflects the vanishing influence of the kinetic component in the diffusive regime, which means $\rho$ naturally converges to $\rho_{\textit{diff}}$ as $\e \to 0$.

{\bf Forward solution approximation}. In the Bi-APNNs, based on \eqref{EI-system}, we design the following two formulations: 
\begin{itemize}
\item \textbf{Explicit formulation (E-bi-APNNs):}
\begin{equation} \label{bi-e}
    \rho_{\text{bi},\theta}^{\mathrm{NN}}(t, \mathbf{x}) := \rho_{\textit{diff}}^{\mathrm{NN}}(t, \mathbf{x}) + \varepsilon \rho_{\textit{corr}}^{\mathrm{NN}}(t, \mathbf{x}),
\end{equation}
\item \textbf{Implicit formulation (I-bi-APNNs):}
\begin{equation} \label{bi-i}
    \rho_{\text{bi},\theta}^{\mathrm{NN}}(t, \mathbf{x}) := \rho_{\textit{diff}}^{\mathrm{NN}}(t, \mathbf{x}) + \rho_{\textit{corr}}^{\mathrm{NN}}(t, \mathbf{x}).
\end{equation}
\end{itemize}

\textbf{BI-APNNs Loss. } Similar to the APNNs, the physics-informed loss is the residual of the micro-macro system \eqref{MM}: 
\begin{equation}
\label{Loss-APNN}
\begin{aligned}
\mathcal{R}_{\text{BI-APNNs}}^{\e} = & \frac{1}{|\mathcal{T} \times \mathcal{D}|} \int_{\mathcal{T}} \int_{\mathcal{D}}\left|\partial_t \rho_{\textit{bi},\theta}^{\mathrm{NN}}+\nabla_{\bx} \cdot\left\langle\bv g_\theta^{\mathrm{NN}}\right\rangle + \red{ \nabla_{\bx}\phi \cdot  \left\langle \nabla_{\bv} g_\theta^{\mathrm{NN}}\right\rangle } \right|^2 \, d\bx d t \\[6pt]
& + \frac{1}{|\mathcal{T} \times \mathcal{D} \times \Omega|} \int_{\mathcal{T}} \int_{\mathcal{D}} \int_{\Omega} \Big| \e^2 \partial_t g_\theta^{\mathrm{NN}}
 + \e(I-\Pi) \left(\bv \cdot \nabla_{\bx} g_\theta^{\mathrm{NN}} + \red{\nabla_{\bx}\phi \cdot \nabla_{\bv} g_\theta^{\mathrm{NN}} }\right) \\[6pt]
&   - 2\bv \cdot \nabla_{\bx}\phi\, \rho_{\textit{bi},\theta}^{\mathrm{NN}}\, M(\bv)  
 + \bv \cdot \nabla_{\bx} \rho_{\textit{bi},\theta}^{\mathrm{NN}} M(\bv)  - \mathcal{Q}( g_\theta^{\mathrm{NN}} ) \Big|^2 \,  d\bv d\bx d t  \\[6pt]
& + \frac{\lambda_1}{| \mathcal{T} \times \partial \mathcal{D} \times \Omega |} \int_{\mathcal{T}}\int_{\gamma_{-}}\left|\mathcal{B}  \left(\rho_{\textit{bi},\theta}^{\mathrm{NN}} M(\bv) + \e g_\theta^{\mathrm{NN}}\right) - f_{\text{BC}}\right|^2 \, ds dt\\[6pt]
 &+\frac{\lambda_2}{|\mathcal{D} \times \Omega|} \int_{\mathcal{D}} \int_{\Omega}\left|\mathcal{I}\left(\rho_{\textit{bi},\theta}^{\mathrm{NN}} M(\bv) +\e g_\theta^{\mathrm{NN}}\right)- f_{\text{IC}} \right|^2
 \, d\bv d\bx.
\end{aligned}
\end{equation}
Note that the difference from the standard APNNs is that here we have two parts, $\rho^{\text{NN}}_{\textit{diff}}$ and $\rho^{\text{NN}}_{\textit{corr}}$, both approximated from neural networks. In the forward problem, we first pre-train a neural network to approximate the solution of the diffusion equation \eqref{Diffusion} and obtain $\rho^{\text{NN}}_{\textit{diff}}$. In this step the loss function is defined as
\begin{equation}
\label{Loss-APNN-Diff}
\begin{aligned}
\mathcal{L}_{\text{diffusion}} &=  
 \frac{1}{|\mathcal{T} \times \mathcal{D}|} \int_{\mathcal{T}} \int_{\mathcal{D}}\left|\partial_t \rho_{\textit{diff},\theta}^{\mathrm{NN}} - \nabla_x \cdot \left( T \nabla_x\rho_{\textit{diff},\theta}^{\mathrm{NN}} - 2 (\rho_{\textit{diff},\theta}^{\mathrm{NN}}\, \nabla_x\phi) \right) \right|^2 \, d\bx d t \\[4pt]
& + \frac{\lambda_1^{\text{diff}}}{| \mathcal{T} \times \partial \mathcal{D}  |} \int_{\mathcal{T}}\int_{\gamma_{-}}\left|\mathcal{B}  \left(\rho_{\textit{diff},\theta}^{\mathrm{NN}}M(\bv) \right) - f_{\text{BC}}\right|^2 \, ds dt\\[4pt]
 &+\frac{\lambda_2^{\text{diff}}}{|\mathcal{D} |} \int_{\mathcal{D}} \int_{\Omega}\left|\mathcal{I}\left(\rho_{\textit{diff},\theta}^{\mathrm{NN}}M(\bv)  \right)- f_{\text{IC}} \right|^2
 \, d\bv d\bx. 
\end{aligned}
\end{equation}
This pretrained model captures the macroscopic behavior of the system when $\varepsilon$ is small. In the subsequent training, $\rho^{\text{NN}}_{\textit{diff}}$ is fixed and treated as a known function, and the neural network only learns the correction term $\rho^{\text{NN}}_{\textit{corr}}$. Incorporating the governing equation of the diffusion limit may accelerate the convergence of the training procedure and potentially improve the accuracy of approximating the solution. Since the macroscopic dynamics are already encoded in a pretrained network, the training process of our designed BI-APNNs is expected to be more efficient. 

The BI-APNNs framework is applied to both the semiconductor Boltzmann equation \eqref{Boltz-eqn} and the Boltzmann-Poisson system \eqref{eqn:BP}. A key distinction in the implementation of these two models lies in the treatment of the electric potential, $\phi(t,x)$. For the standard semiconductor Boltzmann equation, the potential is a known function. In contrast, for the Boltzmann-Poisson system, the potential is itself an unknown that depends on the particle density. To address this, we introduce a separate neural network, $\phi_\theta(t,x)$, dedicated to approximating this self-consistent potential. It is important to note that this architectural approach---employing an auxiliary network for $\phi$ when necessary---is consistently used for both the forward and inverse problems. We summarize the framework of our method for the forward problem in Figure \ref{fig:bi-APNN}.  
 
\begin{figure}[H]
\centering
\includegraphics[width=1\textwidth]{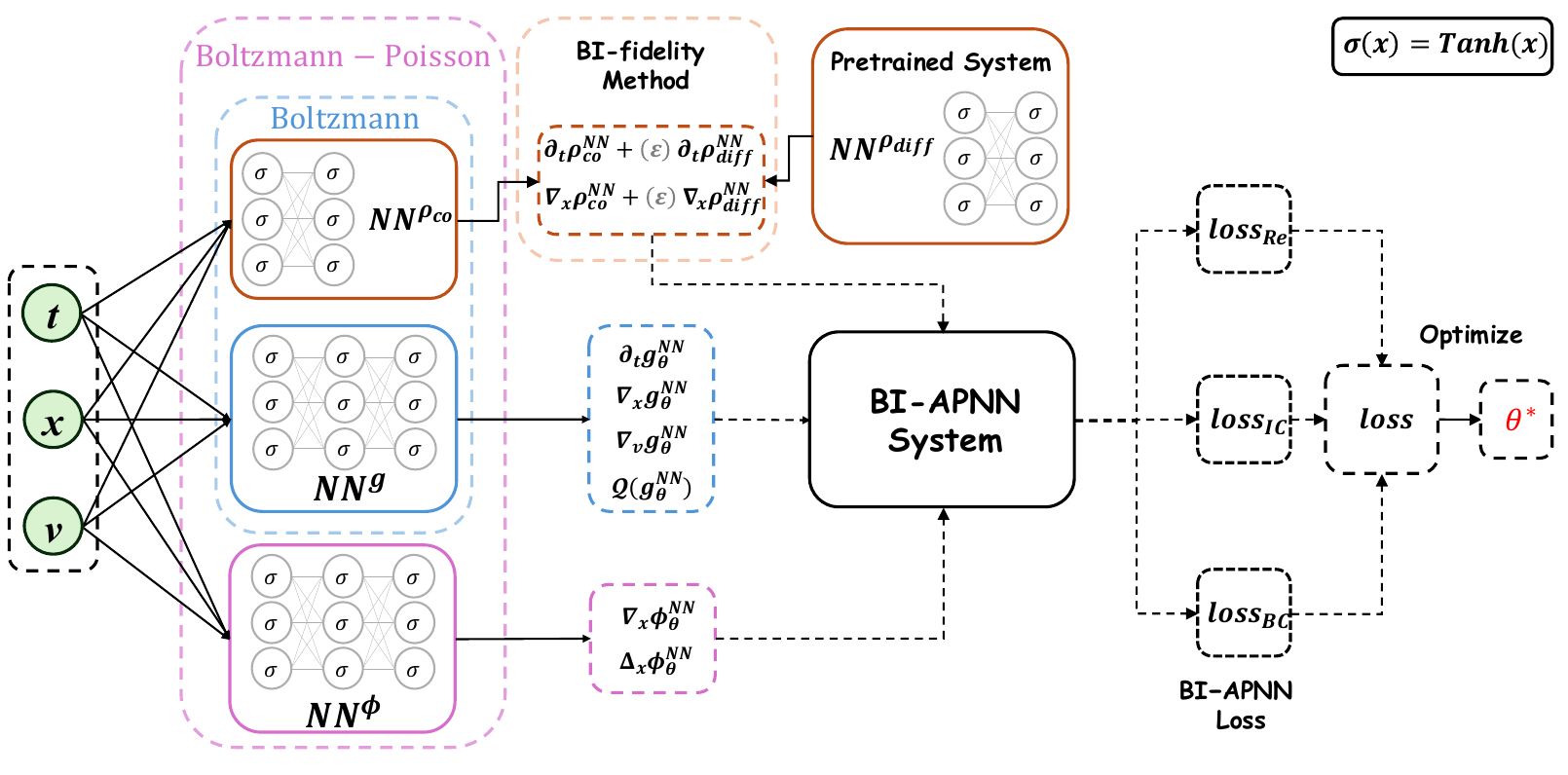}
\caption{Framework of BI-APNNs for the forward problem. Here $\sigma(x)$ denotes the activation function. } 
\label{fig:bi-APNN}
\end{figure}

For the inverse problem, we propose the physics-informed loss $\mathcal{L}_{\text{total}}$, which is composed of three components:
\begin{enumerate}
\item $\mathcal{L}_{\text{BI-APNNs}}$: This loss is the same as in \eqref{Loss-APNN}. 
\item $\mathcal{L}_{\text{diffusion}}$: This loss is the same formulation as \eqref{Loss-APNN-Diff}. However, different from the forward problem, where a pre-trained diffusion model can be used, in the inverse problem setting the key physical parameter (such as $\sigma$) is unknown. Therefore, the diffusion network should be {\it jointly} trained with the Bi-APNN loss, serving as the one regularizer in the total loss.

\item $\mathcal{L}_{\text{data}}$: A data-misfit loss that penalizes discrepancies with observed data.

\begin{equation}
\label{Loss-APNN}
\begin{aligned}
\mathcal{L}_{\text{data}} = 
 \frac{w_d^{\rho}}{N_{d_1}} \sum_{i=1}^{N_{d_1}} \left| \rho_{\textit{bi},\theta}(t_{d}^i, x_{d}^i) -\rho(t_{d}^i,x_{d}^i) \right|^2.
\end{aligned}
\end{equation}

\end{enumerate}
Finally, the total loss is defined by \begin{equation}\mathcal{L}_{\text{total}} = \mathcal{L}_{\text{BI-APNNs}}  + \mathcal{L}_{\text{diffusion}} + \mathcal{L}_{\text{data}}. \end{equation}
Figure \ref{fig:bi-APNN-inverse} shows our method's framework for inverse problem, applied to the semi-conductor Boltzmann model \eqref{Boltz-eqn} and Boltzmann-Poisson system \eqref{eqn:BP}. Further details for the data-matching term are provided in the numerical experiments section.
\begin{figure}[H]
\centering
\includegraphics[width=1\textwidth]{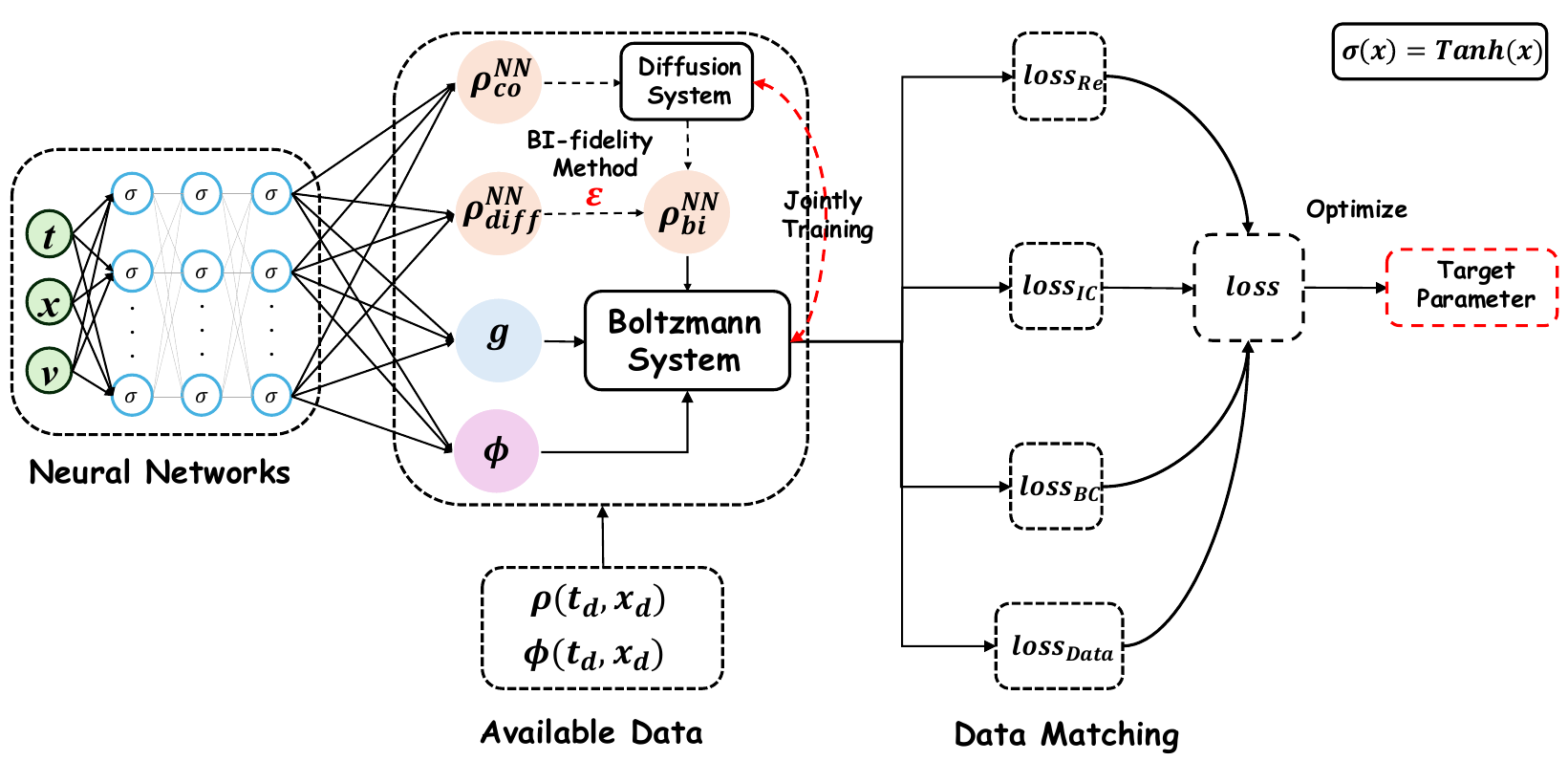}
\caption{Framework of BI-APNNs for the inverse problem.}
\label{fig:bi-APNN-inverse}
\end{figure}

\section{Convergence of the Bi-APNNs solution} \label{sec:analysis}

\subsection{Preliminaries}

We first review an important theorem on the existence of the approximated neural network solution, namely the Universal Approximation Theorem (UAT) \cite{UAT}. We adapt it to our model \eqref{Boltz-eqn} and use the notations introduced in subsection \ref{sec:NN}. 

\begin{lemma}
\label{Lemma1}
Suppose the solution to \eqref{Boltz-eqn} satisfies $f \in C^1([0,T]) \cap C^1(\mathcal{D}) \cap C^1(\Omega)$. Let the activation function $\bar\sigma$ be any non-polynomial function in $C^1(\mathbb R)$, then for any $\delta>0$, there exists a two-layer neural network
$$ f^{NN}(t,x,v) = \sum_{i=1}^{m_1} w_{1i}^{(2)} \bar\sigma\left( \left(w_{i 1}^{(1)}, w_{i 2}^{(1)}, w_{i 3}^{(1)}\right) \cdot(t, x, v)+b_{i}^{(1)}\right) + b_{1}^{(2)}, 
$$
such that 
\begin{equation}
\label{UAT}
\begin{aligned}
\displaystyle &  \left\| f - f^{NN} \right\|_{L^{\infty}(K)} < \delta, \qquad 
\left\| \partial_t ( f - f^{NN}) \right\|_{L^{\infty}(K)} < \delta, \\[4pt]
& \left\| \nabla_x ( f - f^{NN}) \right\|_{L^{\infty}(K)} < \delta, \quad
\left\| \nabla_v ( f - f^{NN}) \right\|_{L^{\infty}(K)} < \delta, 
\end{aligned}
\end{equation}
where the domain $K$ denotes $[0,T]\times\mathcal{D}\times\Omega$. 
\end{lemma}
We remark that the above result can be generalized to a neural network with several hidden layers \cite{Multilayer}.

\subsection{Convergence of the loss function}
Recall the APNNs framework introduced in section \ref{sec:APNN}, the loss function is based on the residual of the macro-micro system \eqref{Macro}--\eqref{Micro}, thus our neural network approximated solutions are $\rho$ and $g$, instead of $f$ if applying PINNs to the model \eqref{Boltz-eqn}. 
In \cite{liu2025asymptotic}, the authors show that a sequence of neural network solutions to 
\eqref{Macro}--\eqref{Micro} exists such that the total loss function converges to zero, if assuming the analytic solution 
\begin{equation}\label{condition}
\rho \in C^1([0,T])\cap C^1(\mathcal{D}), \qquad 
g \in C^1([0,T])\cap C^1(\mathcal{D}) \cap C^1(\Omega). 
\end{equation}

Below we show the convergence of our new APNN approach in the bi-fidelity framework, namely the Bi-APNN method.

\begin{theorem}[Convergence of BI-APNNs]
\label{Thm2}
Let the assumptions of Theorem 1 in \cite{liu2025asymptotic} hold: consider the semiconductor Boltzmann equation \eqref{Boltz-eqn}, where the potential $\phi(t,x)$ is given and satisfies
$\|\nabla_x\phi\|_{L^{\infty}([0,T]\times\mathcal{D})}$ bounded. Let $(\rho,g)$ be the analytic solution of the macro--micro system \eqref{Macro}--\eqref{Micro} which is sufficiently smooth. 

Consider the BI-APNNs' approximations of the macroscopic density by either the explicit formulation
\begin{equation}
\begin{aligned}
        \rho_{\mathrm{bi},\theta}^{\mathrm{NN}}(t,x) = \rho_{\mathrm{diff}}^{\mathrm{NN}}(t,x) + \varepsilon\,\rho_{\mathrm{corr}}^{\mathrm{NN}}(t,x),
\end{aligned}
\tag{E-bi}
\end{equation}
or the implicit formulation
\begin{equation}
\begin{aligned}
        \rho_{\mathrm{bi},\theta}^{\mathrm{NN}}(t,x) = \rho_{\mathrm{diff}}^{\mathrm{NN}}(t,x) + \rho_{\mathrm{corr}}^{\mathrm{NN}}(t,x).
\end{aligned}
\tag{I-bi}
\end{equation}

Let $\mathcal{R}_{\mathrm{BI\text{-}APNNs}}^{\varepsilon}(\rho_\mathrm{bi}, g)$ denote the BI-APNNs total loss function given in \eqref{Loss-APNN}, 
then there exists a sequence of neural network parameters $\{\theta_{[j]}\}_{j=1}^{\infty}$ (where $\theta_{[j]}$ encompasses parameters for both $\rho_{\mathrm{diff}}^{\mathrm{NN}}$ and $\rho_{\mathrm{corr}}^{\mathrm{NN}}$) and corresponding BI-APNNs approximations $\{(\rho_{\mathrm{bi},j}, g_j)\}_{j=1}^{\infty}$
(with $\rho_{\mathrm{bi},j}$ given either by (E-bi) or (I-bi)) such that
\begin{equation}
    \begin{aligned}
        \mathcal{R}_{\mathrm{Bi\text{-}APNNs}}^{\varepsilon}(\rho_{\mathrm{bi},j}, g_j) \;\to\; 0,
        \qquad\text{as } j\to\infty.
    \end{aligned}
\end{equation}
\end{theorem}

\begin{proof}
Expanding the analytic solution $\rho$ in the following form: 
\begin{equation}
\label{asymp-exp}
\rho(t,x) = \rho^{(0)}(t,x) + \rho^{(1)}(t,x;\e),
\end{equation}
where $\rho^{(0)}$ satisfies the limiting drift-diffusion equation \eqref{Diffusion}, and $\rho^{(1)}$ is a higher-order correction term. The micro-part $g$ is also smooth under the assumptions \eqref{condition}.

We construct the BI-APNNs solution sequence as follows. Let $\delta_j^{(1)},\delta_j^{(2)} > 0$ be a sequence such that they converge to zero as $j \to \infty$ (e.g., $\delta_j = 1/j$). By UAT, there exist networks such that:
\begin{align}
\label{approx-bounds}
\|\rho_{\mathrm{diff},j}^{\mathrm{NN}} - \rho^{(0)}\|_{L^\infty(K_1)} &\leq \delta_j^{(1)}, \\
\|g_j^{\mathrm{NN}} - g\|_{L^\infty(K_2)} &\leq \delta_j^{(2)},
\end{align}
where $K_1 = [0,T] \times \mathcal{D}$ and $K_2 = [0,T] \times \mathcal{D} \times \Omega$. Due to smoothness of the solutions, similar bounds hold for derivatives up to first and second orders.

Let $\delta_j = \text{max}(\delta_j^{(1)},\delta_j^{(2)}) $. We consider the two formulations: 
\begin{itemize}
    \item \textbf{Explicit formulation (E-bi):} Assume
    $\|\e\rho_{\mathrm{corr},j}^{\mathrm{NN}} - \rho^{(1)}\|_{L^\infty(K_1)} \leq O(\delta_j)$,
    the Bi-APNNs solution is $\rho_{\mathrm{bi},j} = \rho_{\mathrm{diff},j}^{\mathrm{NN}} + \varepsilon \rho_{\mathrm{corr},j}^{\mathrm{NN}}$, then we have:
    \begin{equation}
            \|\rho_{\mathrm{bi},j} - \rho\|_{L^\infty(K_1)} \leq \|\rho_{\mathrm{diff},j}^{\mathrm{NN}} - \rho^{(0)}\|_{L^\infty(K_1)} +  \|\e\rho_{\mathrm{corr},j}^{\mathrm{NN}} - \rho^{(1)}\|_{L^\infty(K_1)}  \leq  O(\delta_j).
    \end{equation}

    \item \textbf{Implicit formulation (I-bi):} Assume 
    $\|\rho_{\mathrm{corr},j}^{\mathrm{NN}} - \rho^{(1)}\|_{L^\infty(K_1)} \leq O(\delta_j)$, 
    the BI-APNNs solution is $\rho_{\mathrm{bi},j} = \rho_{\mathrm{diff},j}^{\mathrm{NN}} + \rho_{\mathrm{corr},j}^{\mathrm{NN}}$, then we have:
    \begin{equation}
            \|\rho_{\mathrm{bi},j} - \rho\|_{L^\infty(K_1)} \leq \|\rho_{\mathrm{diff},j}^{\mathrm{NN}} - \rho^{(0)}\|_{L^\infty(K_1)} + \|\rho_{\mathrm{corr},j}^{\mathrm{NN}} - \rho^{(1)}\|_{L^\infty(K_1)} \leq  O(\delta_j).
    \end{equation}

\end{itemize}

    
 Similar to  \cite[Theorem 1]{liu2025asymptotic}, there exists a constant $L > 0$ such that
\[
\mathcal{R}_{\mathrm{BI-APNNs}}^{\varepsilon}(\rho_{\mathrm{bi},j}, g_j) \leq L \left( \|\rho_{\mathrm{bi},j} - \rho\|_{L^{\infty}(K_1)} + \|g_j - g\|_{L^{\infty}(K_2)} \right) \leq O(\delta_j), 
\]
where $L$ is independent of $\varepsilon$. 



\end{proof}


\section{Numerical examples} \label{sec:Num}

\input{Numerical_Experiment}

\section{Conclusions and future work} 
In this paper, we have developed and analyzed a Bi-fidelity Asymptotic-Preserving Neural Network (BI-APNNs) framework for solving forward and inverse problems for the semiconductor Boltzmann equation. This work builds upon the standard APNNs methodology, specifically addressing a key practical bottleneck: the slow convergence. The slower convergence of the standard APNNs can be observed in both the training loss and the relative $\ell_2$ error curves, particularly in the computationally challenging fluid-dynamic regime. In contrast, the BI-APNNs shows faster overall convergence of the macroscopic density, $\rho$, in the computationally challenging fluid-dynamic regime. The core of our approach is a novel bi-fidelity decomposition of the density. In the forward problem, it allows a pre-trained low-cost network to capture the dominant physical behavior, while a much smaller network learns the remaining high-fidelity correction. Through a series of numerical experiments, we have demonstrated the key advantages of this bi-fidelity formulation. The results confirm that, particularly in the small-$\varepsilon$ regime, the BI-APNNs not only significantly accelerate training convergence but also improve the accuracy of the forward problem solution. Most importantly, we have shown that this enhanced accuracy translates directly into superior performance for the associated inverse problems when compared to the standard APNNs. This work underscores the potential of bi-fidelity methods to enhance physics-informed neural networks, making them a more efficient and reliable tool for complex multiscale kinetic problems.

Several promising directions for future research emerge from this work. On the theoretical front, our convergence analysis could be further refined by incorporating tools from Barron spaces \cite{Gu-Ng} and developing posterior error estimates to more precisely quantify the network's approximation capabilities. From a practical standpoint, extending the BI-APNNs framework to higher-dimensional problems is a key next step. Furthermore, the computational efficiency demonstrated by the BI-APNNs makes them an particularly attractive candidate for uncertainty quantification (UQ) problems, where the need for numerous forward solves often creates a computational bottleneck.

\section*{Acknowledgement}
\label{sec:ack}
L.~Liu acknowledges the support by National Key R\&D Program of China (2021YFA1001200), Ministry of Science and Technology in China, General Research Fund (14301423 \& 14307125) funded by Research Grants Council of Hong Kong. X. Zhu was supported by the Simons Foundation (MPS-TSM-00007740). 

\section{Appendix}
\label{sec:Appendix}
\input{Appendix}

\bibliographystyle{siam}
\bibliography{APNN.bib}

\end{document}

%% file: Numerical_Experiment.tex
In our numerical experiments, we will show several examples to illustrate the robustness and accuracy of our designed BI-APNNs method. We compare the reference solution obtained by traditional numerical methods, APNNs method, PINNs method and BI-APNNs methods in different regimes, ranging from the kinetic regime ($\e\approx O(1)$) to the diffusive regime ($\e \ll 1$), and observe that the BI-APNNs can capture the multiscale nature of the model thanks to the design of the loss function based on the macro-micro decomposition.

For the velocity discretization, the integral in $v$ is computed by using the Hermite quadrature rule with $N_v$ quadrature points. We adopt the Hermite polynomials in velocity discretization to provide more accurate velocity derivatives, a technique similar to the moment method \cite{JP00,SZ99}. See details in the Appendix. 

\textbf{Networks Architecture.} In our experiments, we use the feed-forward neural network (FNN) with one input layer, one output layer and $4$ hidden layers with $128$ neurons in each layer, unless otherwise specified. The hyperbolic tangent function (Tanh) is chosen as our activation function. This setting is adopted from previous work \cite{APNN-transport,Jin-Ma-Wu1}, as it has been validated to perform effectively on similar Boltzmann problems.


\textbf{Training Settings.} We train the neural network by Adam \cite{kingma2014adam} with Xavier initialization. We set epochs to be $20000$ and the learning rate to be $10^{-4}$, and use full batch for most of the following experiments in the numerical experiments unless otherwise specified. All the hyper-parameters are chosen by trial and error.  

\textbf{Loss Design.}  To compute the empirical risk losses for BI-APNNs in \eqref{empirical_Loss-APNNs} and \eqref{empirical_Loss-biAPNN}. For most of our experiments, we consider the spatial and temporal domains to be $[0,1]$ and $[0,0.1]$ respectively.  We choose the collocation points $\{(t_i, x_i, v_i)\}$ for $f(t,x,v)$ in the following way. For spatial points $x_i$, we select $99$ interior points evenly spaced in $[0, 1]$.  For temporal points $t_i$, we select $20$ interior points evenly spaced in the range $[0,0.1]$. We use the tensor product grid for the collocation points. For velocity points $v_i$, $N_v = 8$ points are generated by the Hermite quadrature rule.  The penalty parameters in \eqref{empirical_Loss-APNNs} and \eqref{empirical_Loss-biAPNN} are set to $(\lambda_1, \lambda_2) = (1,1)$. 

{\bf Performance metric}. The reference solutions are obtained by traditional AP scheme \cite{JP2000}. We compute the relative $\ell^2$  error of mean and standard deviation of the density $\rho(t,x)$ between our proposed neural network approximations and reference solutions, with the relative $\ell^2$ error defined by: 
\begin{equation}
\mathcal{E}(t):=\sqrt{\frac{\sum_j|\rho_{\theta}^{\text{NN}}(t,x_j)-\rho^{\text{ref}}(t,x_j)|^2}{\sum_j|\rho^{\text{ref}}(t,x_j)|^2}}.
\end{equation}

We run our experiments on a server with Intel(R) Xeon(R) Gold $6230$ and GPU $A40$.

\subsection{Problem I: Bi-APNNs for deterministic problem}
We first consider the following setup for the model equation \eqref{Boltz-eqn}:
$$ x\in [0,1], \qquad \phi = e^{-50 exp(1) (1/4-x)^2}, \qquad \sigma(v,w)=2. $$
Assume the initial data $\displaystyle f(t=0,x,v) = \frac{1}{\sqrt{\pi}} e^{-v^2}$ and incoming boundary conditions in space. 

\subsubsection{Forward Problem}
In the forward problem, we compare the performance of PINNs against several APNNs variants across different values of $\varepsilon$. To ensure a fair comparison, all models share an identical network architecture: four hidden layers with 128 neurons each. This configuration is adopted from previous work \cite{APNN-transport,Jin-Ma-Wu1} due to its established effectiveness on similar problems. 

Figure~\ref{fig:test1_compare_acc} compares the final density solution, $\rho(x, T=0.1)$, generated by each method against a traditional Asymptotic-Preserving (AP) scheme reference solution \cite{JP2000} for various $\varepsilon$. For a large epsilon, such as $\varepsilon=1$, all methods yield solutions that closely match the reference solution. However, as $\varepsilon$ decreases, the performance of standard PINNs degrades significantly. Notably, for $\varepsilon \ll 1$, the PINNs solution collapses to a trivial constant, a known failure mode. In this regime, two BI-APNNs maintain higher accuracy than the standard APNNs. To quantitatively validate these observations, we report the relative $\ell^2$ errors in Table~\ref{tab:test1_bi_forward}. The results confirm that BI-APNNs achieve the lowest error among all tested methods, particularly when $\varepsilon$ is small.

\begin{figure}[H]
    \centering
    \includegraphics[width=0.47\linewidth]{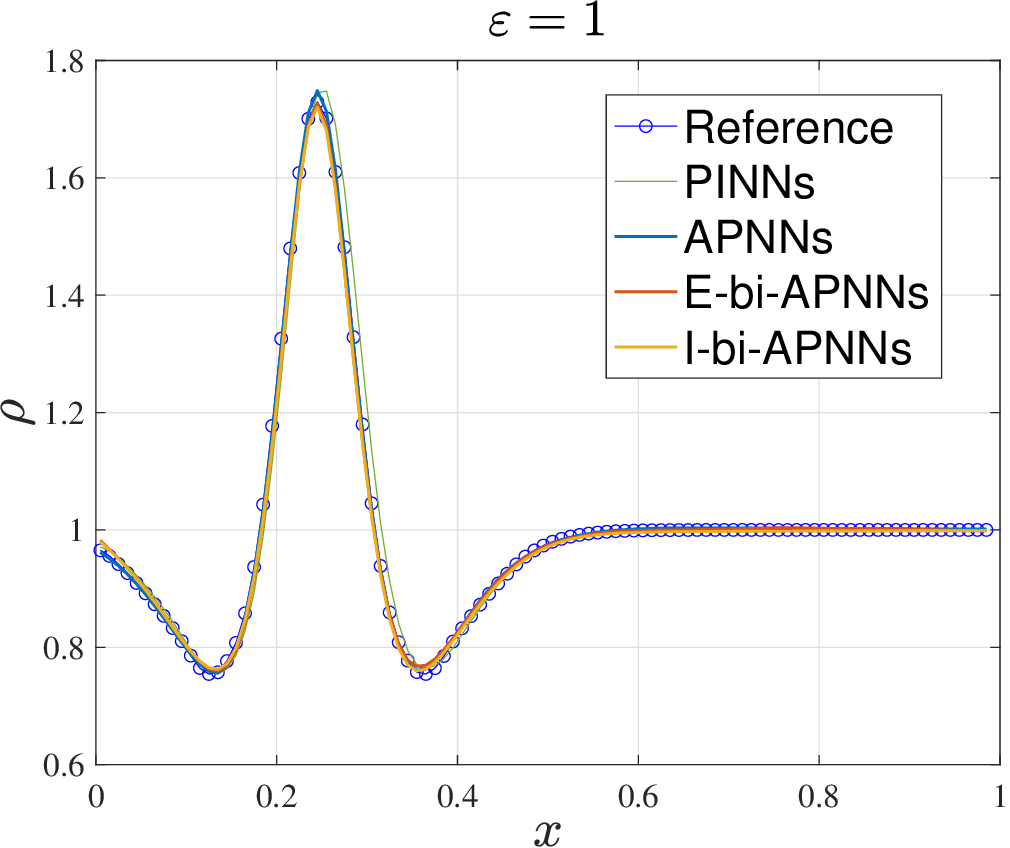}
     \includegraphics[width=0.47\linewidth]{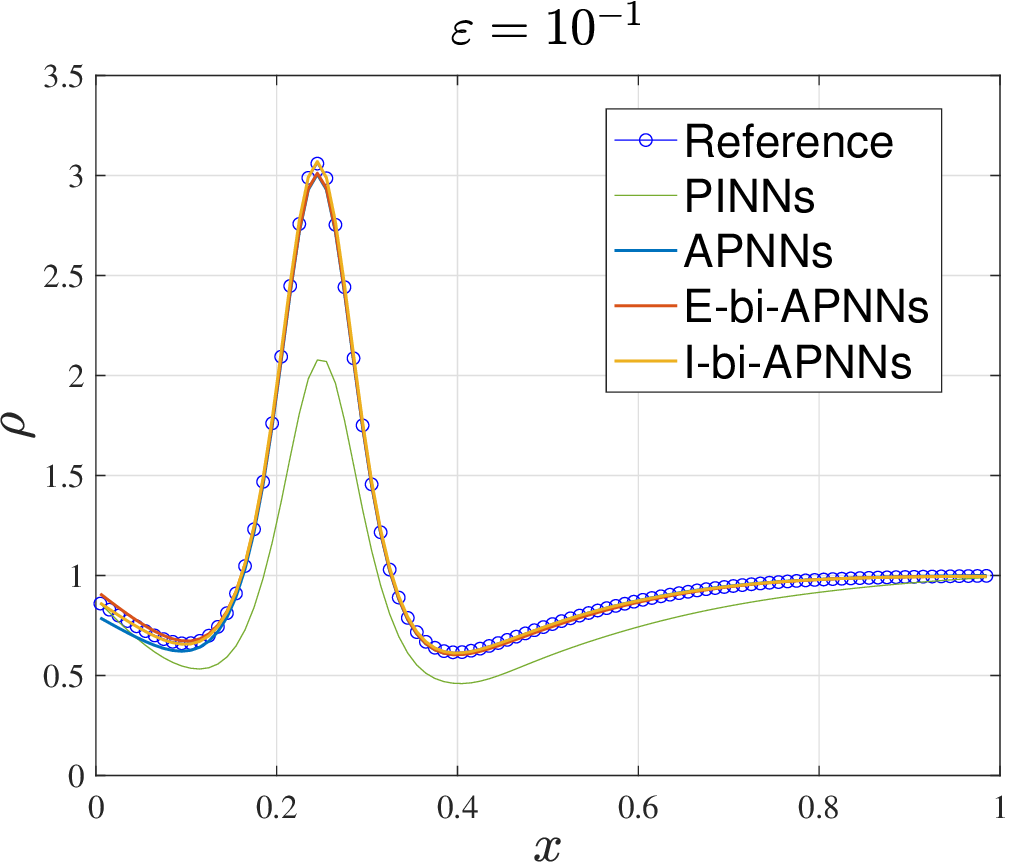}
    \includegraphics[width=0.47\linewidth]{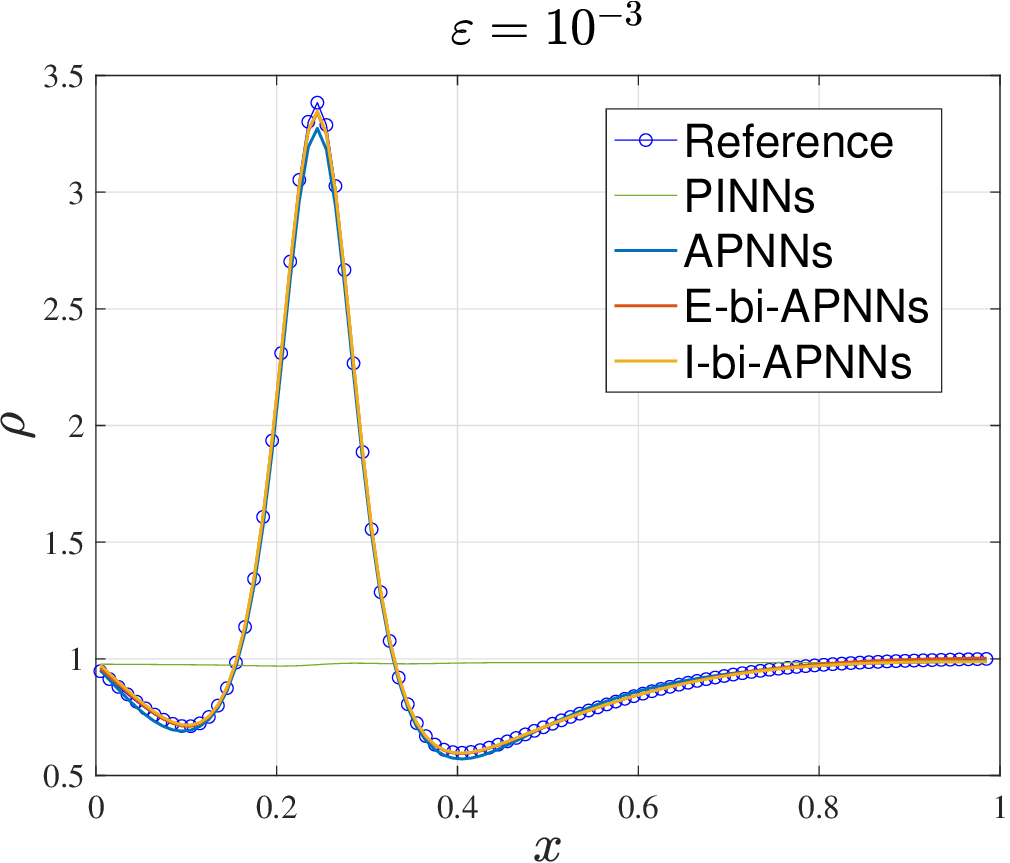}
    \includegraphics[width=0.47\linewidth]{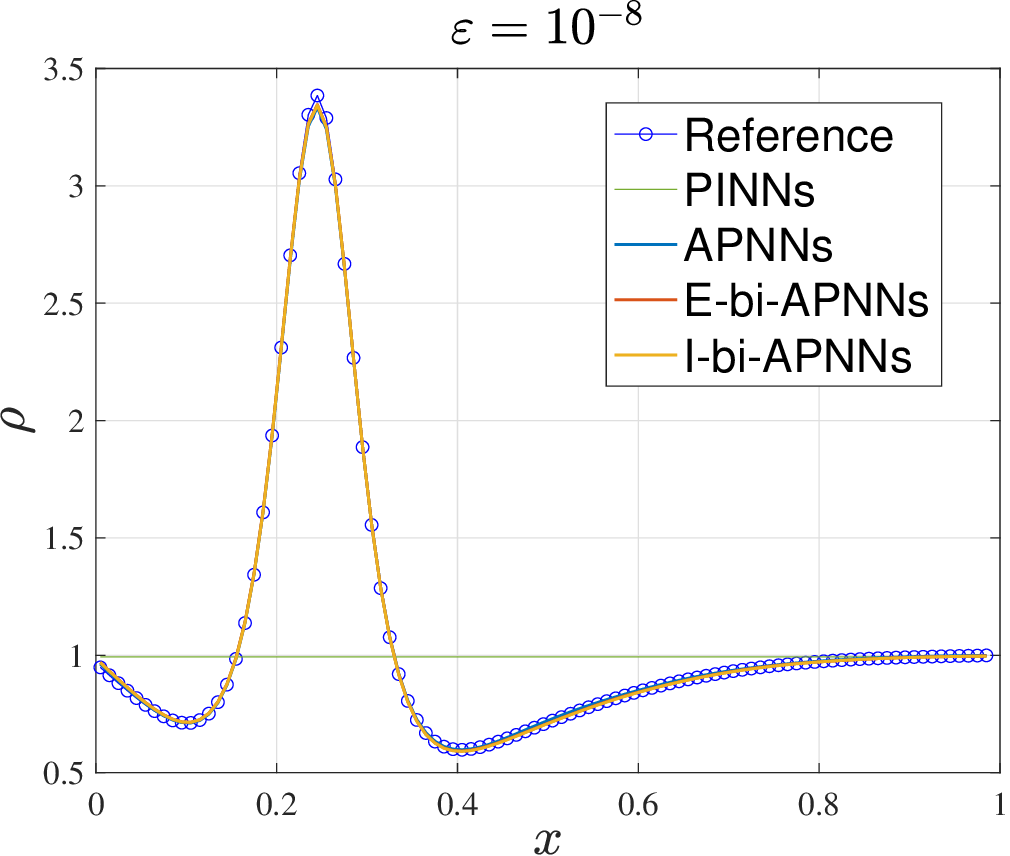}
    \caption{Problems I with different $\e$. 
Density $\rho$ for PINNs, APNNs, BI-APNNs, and reference solutions at $T=0.1$.}
    \label{fig:test1_compare_acc}
\end{figure}


\begin{table}[!htbp]
    \centering
    \begin{tabular}{l|c|c|c|c} 
        \toprule
        Method & \multicolumn{4}{c}{Relative $\ell^2$ error} \\
        \cmidrule(lr){2-5}
        & $\e=1$ & $\e=10^{-1}$ & $\e=10^{-3}$ & $\e=10^{-8}$ \\
        \midrule
        PINNs & $2.28 \times 10^{-2}$ & $2.51 \times 10^{-1}$ & $5.07 \times 10^{-1}$ & $5.04 \times 10^{-1}$ \\
        APNNs & $\mathbf{1.08 \times 10^{-2}}$ & $1.94 \times 10^{-2}$ & $2.38 \times 10^{-2}$ & $1.06 \times 10^{-2}$ \\
        E-bi-APNNs & $1.16\times 10^{-2}$ & $1.33 \times 10^{-2}$ & $\mathbf{7.15 \times 10^{-3}}$ & $\mathbf{7.21 \times 10^{-3}}$ \\
        I-bi-APNNs & $1.19 \times 10^{-2}$ & $\mathbf{5.65 \times 10^{-3}}$ & $7.65 \times 10^{-3}$ & $7.64 \times 10^{-3}$ \\
        \bottomrule
    \end{tabular}
    \caption{Problems I. Relative $\ell^2$ error comparison for different methods with different $\e$ at the final time at $T=0.1$. }
    \label{tab:test1_bi_forward}
\end{table}

Figure~\ref{fig:BI_AP_training} compares the training convergence of BI-APNNs and APNNs across various $\varepsilon$ values under identical network settings. The loss curves clearly show that BI-APNNs converge significantly faster, an advantage that is particularly pronounced as $\varepsilon$ becomes small.

\begin{figure}[H]
    \centering
    \includegraphics[width=0.49\linewidth]{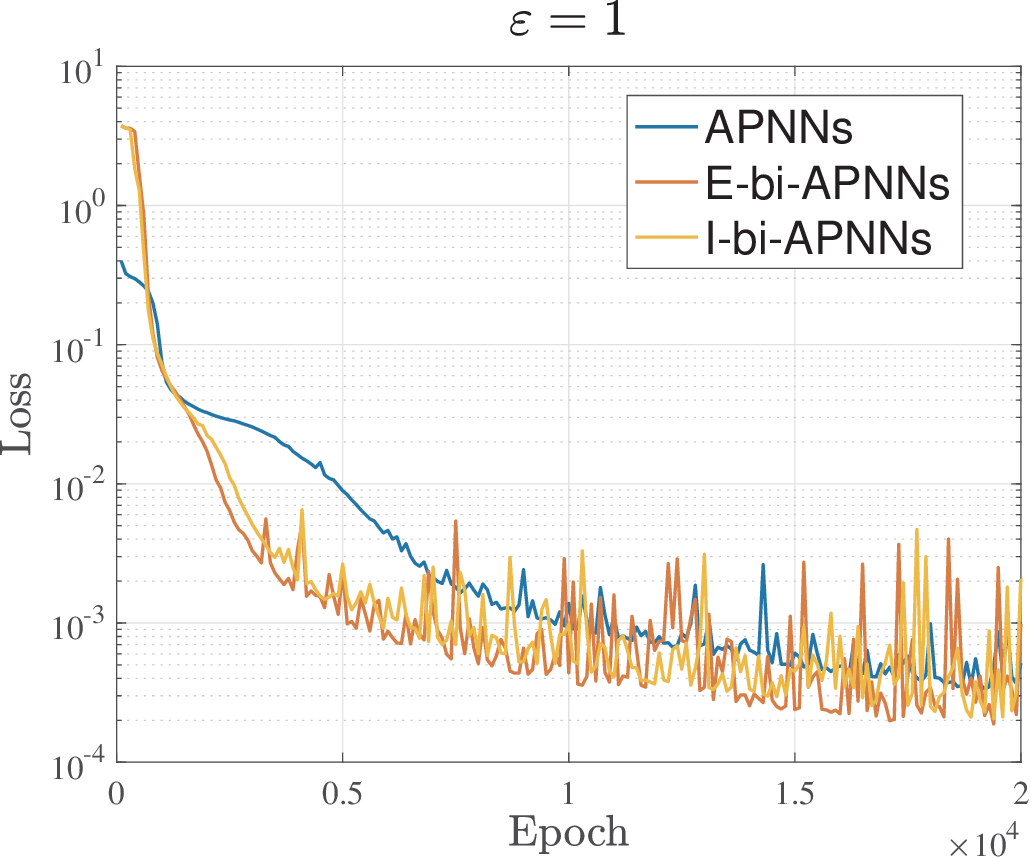}
    \includegraphics[width=0.49\linewidth]{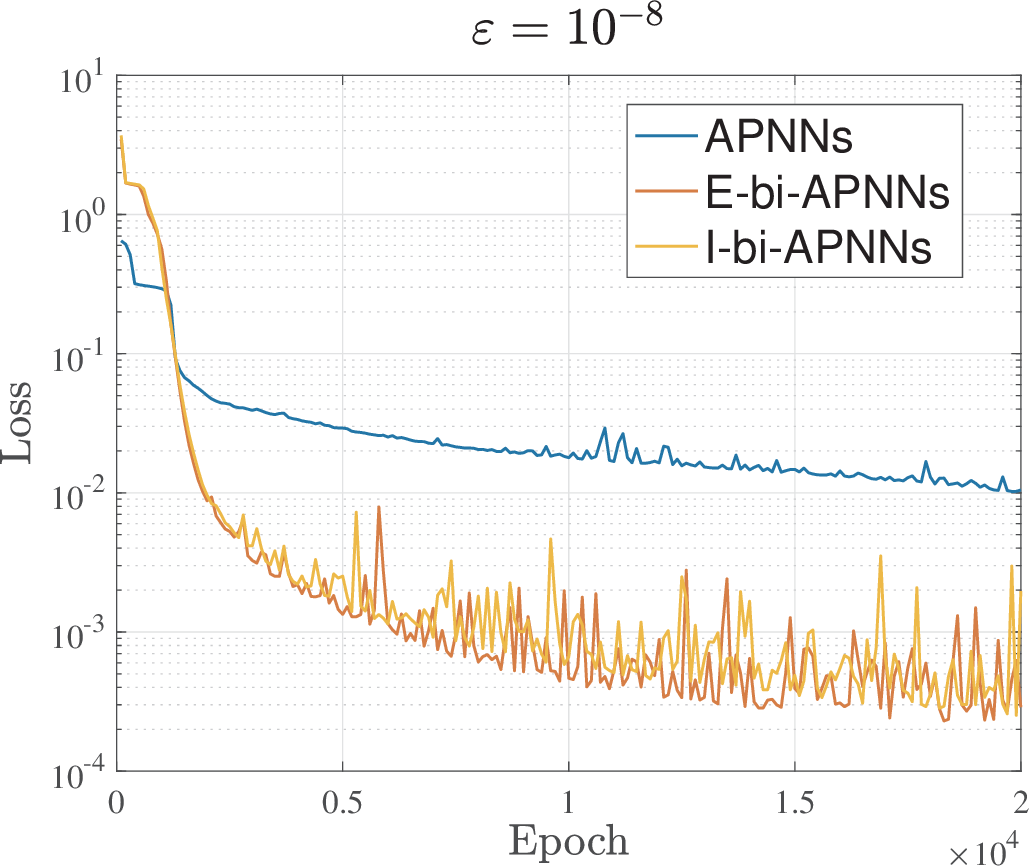}
    \caption{Problem I with $\e=1$ and $\e = 10^{-8}$. Comparison of training losses between Bi-APNNs and APNNs. 
    }
    \label{fig:BI_AP_training}
\end{figure}

Figure~\ref{fig:BI_AP_relative_test1} presents the convergence of the relative $\ell^2$ error for both BI-APNNs and APNNs across different $\varepsilon$ values. The figures clearly demonstrate that our proposed BI-APNNs converge substantially faster thanks to the pretrained diffusion term, which provides a reliable initial approximation and guides the network toward faster convergence, an advantage that is most evident in the small $\varepsilon$ regime.

\begin{figure}[H]
    \centering
    \includegraphics[width=0.49\linewidth]{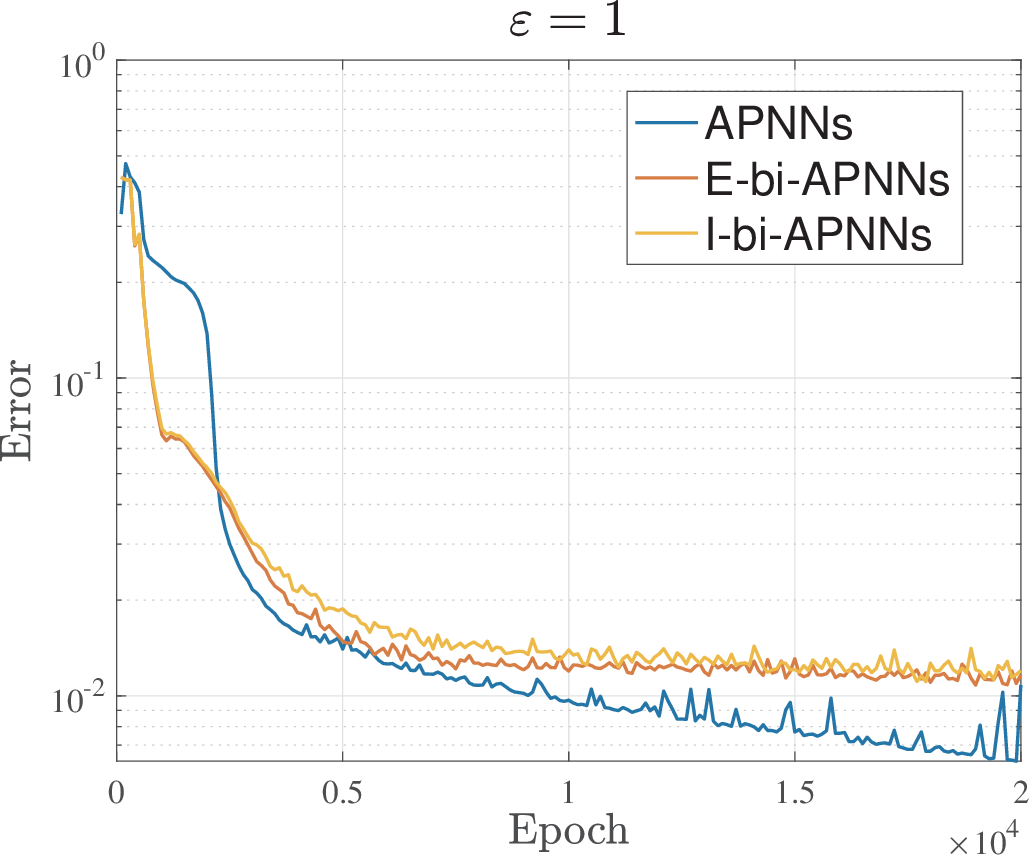}
    \includegraphics[width=0.49\linewidth]{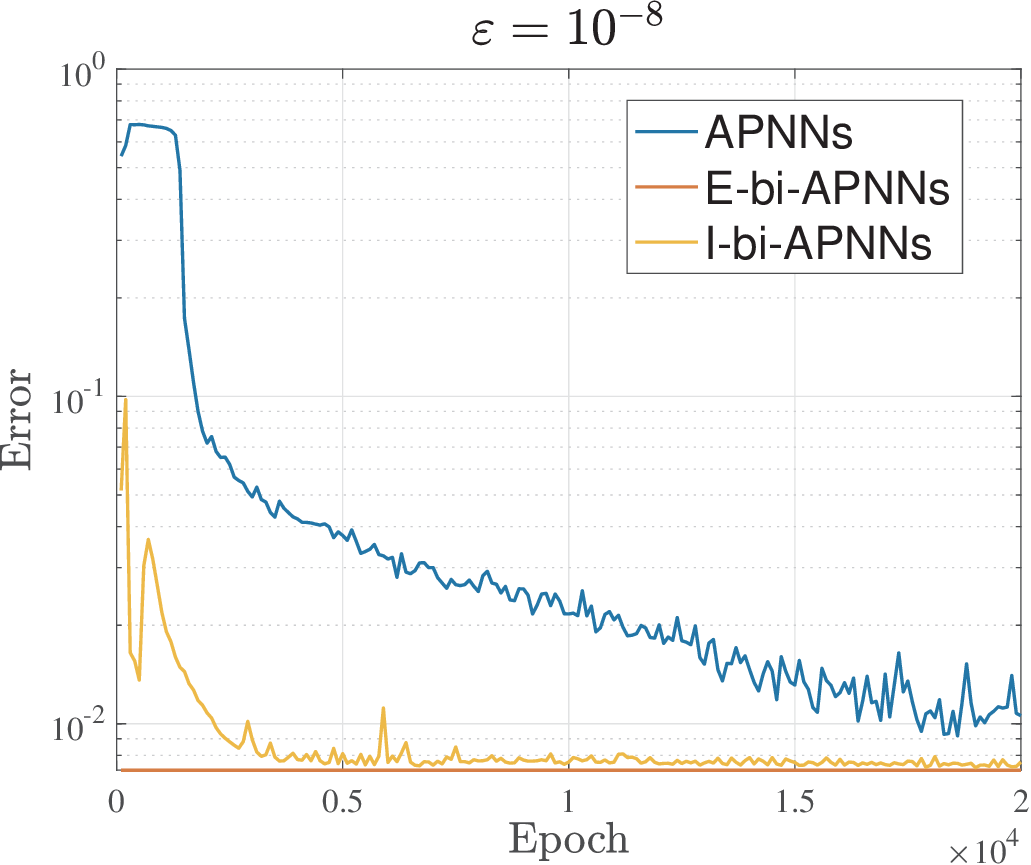}
    \caption{Problem I with $\e=1$ and $\e = 10^{-8}$. Comparison of convergence in relative errors of Bi-APNNs and APNNs.}
    \label{fig:BI_AP_relative_test1}
\end{figure}



To further demonstrate the advantages of our methods in training efficiency, we vary the network depth and width for the pretrained diffusion system \eqref{Loss-diffusion_empirical}, APNNs, and BI-APNNs models. The results for the pretrained diffusion equation presented in Table~\ref{tab:test1_bi_fidelity_diffusion}, indicate that a shallow network with just one to two hidden layers is sufficient to obtain errors in the $O(10^{-2}) \sim O(10^{-3})$ range. We also noted that the neuron count did not significantly impact computation time, so we fixed this parameter in subsequent experiments. For the APNNs and BI-APNNs models, the results for different scales, $\varepsilon = 1$ and $\varepsilon = 10^{-8}$, are detailed in Tables~\ref{tab:test1_bi_fidelity_ep1} and~\ref{tab:test1_bi_fidelity_ep1e-8}. Compared to BI-APNNs,  the standard APNNs often require a deeper architecture of three to four hidden layers to achieve comparable accuracy. Notably, in the small $\varepsilon$ regime, the network for the correction term, $\rho^{\text{NN}}_{\textit{corr}}$, captures the small-scale phenomena and achieves an error of $O(10^{-3})$ even without any hidden layers. This result demonstrates our method's strong performance in significantly reducing computational resources.
    

\begin{table}[!htbp]
    \centering
    \scalebox{1}{
    \begin{tabular}{c|c|c|c}
        \toprule
        \multirow{2}{*}{Layer / Neuron} & \multicolumn{3}{c}{Relative Error} \\
        \cmidrule(lr){2-4}
        & 32 & 64 & 128 \\
        \midrule
        0 & $6.54 \times 10^{-2}$ & $6.40 \times 10^{-2}$ & $5.68 \times 10^{-1}$ \\
        1 & $1.01 \times 10^{-2}$ & $9.06 \times 10^{-3}$ & $7.09 \times 10^{-3}$ \\
        2 & $5.86 \times 10^{-3}$ & $5.51 \times 10^{-3}$ & $9.08 \times 10^{-3}$ \\
        3 & $7.61 \times 10^{-3}$ & $8.84 \times 10^{-3}$ & $5.68 \times 10^{-3}$ \\
        4 & $7.07 \times 10^{-3}$ & $5.24 \times 10^{-3}$  & $6.85 \times 10^{-3}$ \\
        \bottomrule
    \end{tabular}
    } 
    
    \caption{Problems I. Comparison of the relative error for the diffusion Equation under different layers and neurons.}
    \label{tab:test1_bi_fidelity_diffusion}
\end{table}

\begin{table}[!htbp]
    \centering
    \scalebox{0.85}{
    \begin{tabular}{c|c|c|c|c}
        \toprule
        \multirow{2}{*}{Method} & \multirow{2}{*}{Layer / Neuron} & \multicolumn{3}{c}{Relative Error} \\
        \cmidrule(lr){3-5}
        & & 32 & 64 & 128 \\
        \midrule
        \multirow{5}{*}{\centering APNNs} 
        & 0 & $1.95 \times 10^{-1}$ & $1.96 \times 10^{-1}$ & $1.96 \times 10^{-1}$ \\
        & 1 & $2.67 \times 10^{-2}$ & $2.82 \times 10^{-2}$ & $2.39 \times 10^{-2}$ \\
        & 2 & $5.94 \times 10^{-3}$ & $4.00 \times 10^{-3}$ & $5.40 \times 10^{-3}$ \\
        & 3 & $3.29 \times 10^{-3}$ & $4.36 \times 10^{-3}$ & $4.32 \times 10^{-3}$ \\
        & 4 & $3.54 \times 10^{-3}$ & $4.78 \times 10^{-3}$ & $6.19 \times 10^{-3}$ \\
        \midrule
        \multirow{5}{*}{\centering Bi-APNNs-Explicit} 
        & 0 & $4.12 \times 10^{-1}$ & $4.14 \times 10^{-1}$ & $4.15 \times 10^{-1}$ \\
        & 1 & $1.50 \times 10^{-1}$ & $9.41 \times 10^{-2}$ & $3.51 \times 10^{-2}$ \\
        & 2 & $2.62 \times 10^{-2}$ & $1.14 \times 10^{-2}$ & $1.37 \times 10^{-2}$ \\
        & 3 & $9.33 \times 10^{-3}$ & $7.97 \times 10^{-3}$ & $1.18 \times 10^{-2}$ \\
        & 4 & $1.71 \times 10^{-2}$ & $9.31 \times 10^{-3}$ & $1.05 \times 10^{-2}$ \\
        \midrule
        \multirow{5}{*}{\centering Bi-APNNs-Implicit} 
        & 0 & $4.11 \times 10^{-1}$ & $4.13 \times 10^{-1}$ & $4.15 \times 10^{-1}$ \\
        & 1 & $1.36 \times 10^{-1}$ & $9.56 \times 10^{-2}$ & $3.71 \times 10^{-2}$ \\
        & 2 & $3.51 \times 10^{-2}$ & $1.42 \times 10^{-2}$ & $2.13 \times 10^{-2}$ \\
        & 3 & $1.80 \times 10^{-2}$ & $1.21 \times 10^{-2}$ & $1.28 \times 10^{-2}$ \\
        & 4 & $1.38 \times 10^{-2}$ & $1.27 \times 10^{-2}$ & $1.10 \times 10^{-2}$ \\
        \bottomrule
    \end{tabular}
    } 
    \caption{Problems I. Comparison of Relative Error for APNN, bi-APNN-implicit, and bi-APNN-explicit methods under different layers and neurons, with $\varepsilon=1$.}
    \label{tab:test1_bi_fidelity_ep1}
\end{table}

\begin{table}[!htbp]
    \centering
    \scalebox{0.85}{
    \begin{tabular}{c|c|c|c|c}
        \toprule
        \multirow{2}{*}{Method} & \multirow{2}{*}{Layer / Neuron} & \multicolumn{3}{c}{Relative Error} \\
        \cmidrule(lr){3-5}
        & & 32 & 64 & 128 \\
        \midrule
        \multirow{5}{*}{\centering APNNs} 
        & 0 & $5.04 \times 10^{-1}$ & $5.21 \times 10^{-1}$ & $5.11 \times 10^{-1}$ \\
        & 1 & $1.06 \times 10^{-1}$ & $1.23 \times 10^{-2}$ & $8.91 \times 10^{-3}$ \\
        & 2 & $1.05 \times 10^{-2}$ & $8.64 \times 10^{-3}$ & $1.37 \times 10^{-2}$ \\
        & 3 & $8.56 \times 10^{-3}$ & $9.46 \times 10^{-3}$ & $1.12 \times 10^{-2}$ \\
        & 4 & $6.89 \times 10^{-3}$ & $1.19 \times 10^{-2}$ & $7.58 \times 10^{-3}$ \\
        \midrule
        \multirow{5}{*}{\centering Bi-APNNs-Explicit} 
        & 0 & $7.21 \times 10^{-3}$ & $7.21 \times 10^{-3}$ & $7.21 \times 10^{-3}$ \\
        & 1 & $7.21 \times 10^{-3}$ & $7.21 \times 10^{-3}$ & $7.21 \times 10^{-3}$ \\
        & 2 & $7.21 \times 10^{-3}$ & $7.21 \times 10^{-3}$ & $7.21 \times 10^{-3}$ \\
        & 3 & $7.21 \times 10^{-3}$ & $7.21 \times 10^{-3}$ & $7.21 \times 10^{-3}$ \\
        & 4 & $7.21 \times 10^{-3}$ & $7.21 \times 10^{-3}$ & $7.21 \times 10^{-3}$ \\
        \midrule
        \multirow{5}{*}{\centering Bi-APNNs-Implicit} 
        & 0 & $7.02 \times 10^{-3}$ & $7.35 \times 10^{-3}$ & $7.38 \times 10^{-3}$ \\
        & 1 & $7.22 \times 10^{-3}$ & $6.88 \times 10^{-3}$ & $6.78 \times 10^{-3}$ \\
        & 2 & $7.06 \times 10^{-3}$ & $6.94 \times 10^{-3}$ & $7.01 \times 10^{-3}$ \\
        & 3 & $6.77 \times 10^{-3}$ & $7.10 \times 10^{-3}$ & $7.12 \times 10^{-3}$ \\
        & 4 & $7.15 \times 10^{-3}$ & $7.06 \times 10^{-3}$ & $7.32 \times 10^{-3}$ \\
        \bottomrule
    \end{tabular}
    } 
    \caption{Problems I. Comparison of Relative Error for APNN, bi-APNN-explicit, and bi-APNN-implicit methods under different layers and neurons, with $\varepsilon=10^{-8}$.}
    \label{tab:test1_bi_fidelity_ep1e-8}
\end{table}

\subsubsection{Inverse problem}
In this section, we will examine the the proposed Bi-APNNs framework for  an inverse problem: inferring the scattering coefficient from measurement data. To assess the robustness and practical utility of our method, we focus on a particularly challenging scenario where only partial observation data is available. The performance of the Bi-APNNs is evaluated across different physical regimes. For a direct comparison, we also solve the same inverse problem using the standard APNNs approach as a baseline, allowing for a clear assessment of the improvements offered by the bi-fidelity formulation.

In real-world applications, the microscopic term $g$ is difficult to measure, we assume that measurement data can only be obtained for the macroscopic density, $\rho$. We generate a synthetic dataset, denoted as $\{\left(t_d^i, x_d^i\right), \rho_{\text{obs}}^i\}_{i=1}^{N_d}$, using the classical AP numerical scheme developed in \cite{JP2000}.

To solve this inverse problem using either the BI-APNNs or APNNs framework, we add a data misfit loss as below to the original empirical loss function \eqref{empirical_Loss-APNNs} and \eqref{empirical_Loss-biAPNN}.   This term, which we denote as $\mathcal{L}_{\text{data}}$, quantifies the misfit between the network's prediction for the density and the observed data via the mean squared error: 
\begin{equation}
    \mathcal{L}_{\text{data}}^{\rho}(\theta)^ = \frac{w_d^{\rho}}{N_{d_1}} \sum_{i=1}^{N_{d_1}} \left| \rho^{\mathrm{NN}}(t_{d}^i, x_{d}^i; \theta) -\rho(t_{d}^i,x_{d}^i) \right|^2.
    \label{Inverse_APNN}
\end{equation}

Since PINNs method can not infer $\sigma$ when $\e$ is small even using full observation data \cite{liu2025asymptotic}, here we only test Bi-APNNs and APNNs' performance on inferring the target parameter $\sigma$ when $\e$ is small. In the inverse problem, we use the same network settings as in the forward problem. We use $100$ random measurements for  $\rho(t_d^i,x_d^i$) to estimate the target parameter $\sigma=2$ under different initial guesses $\sigma_0 = 0.5, 1.0, 1.5, 1.7. 1.9$ under both $\varepsilon = 10^{-3}$ and $\varepsilon = 10^{-8}$. While a fixed learning rate works well for the forward problem, the ill-posed nature of the inverse problem requires a more careful optimization strategy to ensure stable convergence. A high learning rate near the true parameter may cause oscillations and lead to suboptimal convergence. For the inverse problem, we employ the Adam optimizer, and we choose the learning rate decay schedule to be step decay. Specifically, we use an initial learning rate of $10^{-4}$ with a multiplicative decay factor of $0.8$. In addition, to maintain training stability, the learning rate is not allowed to fall below a minimum value of $10^{-6}$.

The proposed BI-APNNs frameworks demonstrate a notable improvement in performance for the inverse problem compared to the standard APNNs. This is evidenced by the convergence behavior of the estimated target parameter, $\sigma$, as illustrated in Figure~\ref{fig:test1_inverse}. Figure~\ref{fig:test1_inverse} displays the mean of the estimated parameter $\sigma$, along with its corresponding confidence interval, computed over five different initial guesses $\sigma_0$.  For a quantitative comparison, Tables~\ref{tab:test1_ep1e-3} and \ref{tab:test1_ep1e-8} present the relative errors between the estimated parameter $\hat{\sigma}$ and the true parameter $\sigma$ for both methods. The results clearly indicate that the proposed two BI-APNNs formulations yield more accurate parameter estimations with much less variability, confirming the advantage of our approach. 
\begin{figure}[H]
    \centering
    \includegraphics[width=0.49\linewidth]{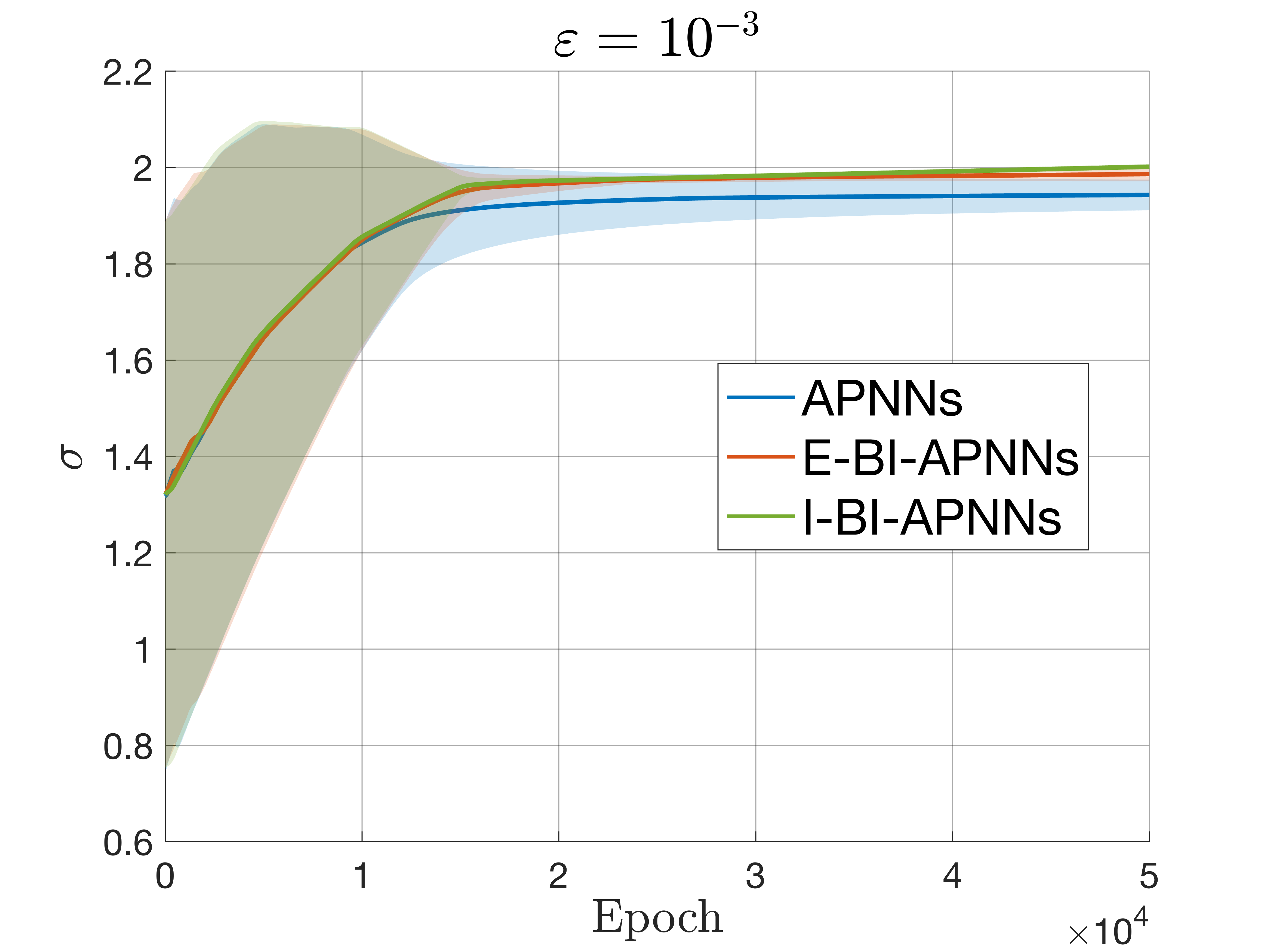}
    \includegraphics[width=0.49\linewidth]{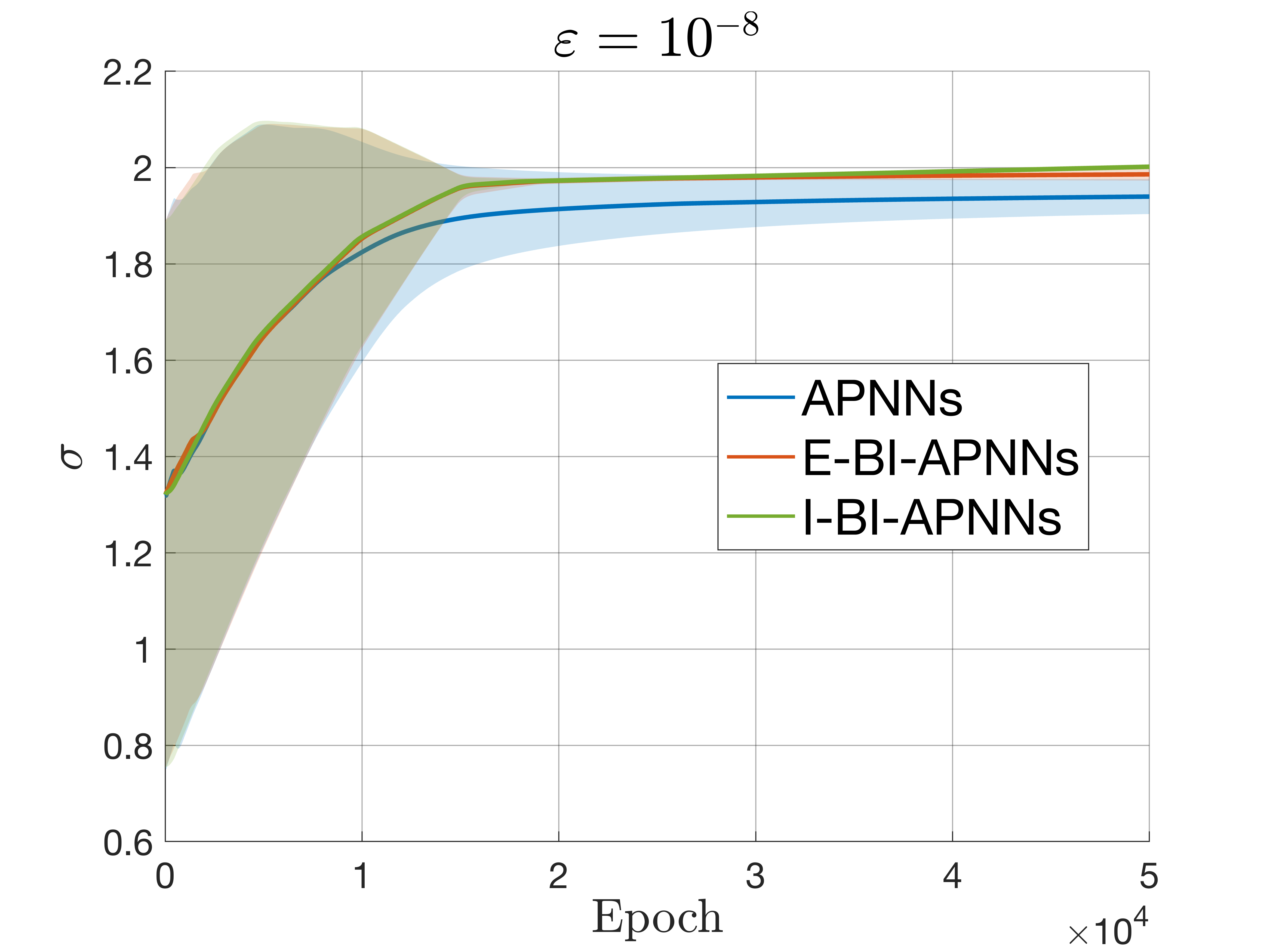}
    \caption{Inverse problem for Problem I. Comparison between three methods under different initial guesses with $\e = 10^{-3}$ and $\e = 10^{-8}$.}
    \label{fig:test1_inverse}
\end{figure}

\begin{table}[H]
    \centering
    \vspace{-15pt}
            \renewcommand{\multirowsetup}{\centering}
            \setlength{\tabcolsep}{5.5pt}
            \scalebox{1}{
                \begin{tabular}{l|c|c|c|c|c|c}
                    \toprule
                    \multirow{2}{*}{Method} & \multicolumn{6}{c}{Initial Guess ($\sigma_0$)} \\
                    \cmidrule(lr){2-7}
                    & 0.5 & 1.0 & 1.5 & 1.7 & 1.9 & Average \\
                    \midrule
                    APNNs & 1.893 & 1.931 & 1.948 & 1.946 & 1.947 & 1.933 \\
                    Bi-APNNs-Explicit & 1.986 & 1.984 & 1.966 & 1.990 & 2.008 & 1.987 \\
                    Bi-APNNs-Implicit & 1.994 & 1.998 & 2.004 & 2.006 & 2.007 & \textbf{2.001} \\
                    \bottomrule
                \end{tabular}
        }
            \caption{Problem I inverse problem. Comparison between three methods under different initial guesses with $\e = 10^{-3}$.}
    \label{tab:test1_ep1e-3}
\end{table}

\begin{table}[H]
    \centering
    \vspace{-15pt}
            \renewcommand{\multirowsetup}{\centering}
            \setlength{\tabcolsep}{5.5pt}
            \scalebox{1}{
                \begin{tabular}{l|c|c|c|c|c|c}
                    \toprule
                    \multirow{2}{*}{Method} & \multicolumn{6}{c}{Initial Guess ($\sigma_0$)} \\
                    \cmidrule(lr){2-7}
                    & 0.5 & 1.0 & 1.5 & 1.7 & 1.9 & Average \\
                    \midrule
                    APNNs & 1.892 & 1.910 & 1.945 & 1.944 & 1.949 & 1.928 \\
                    Bi-APNNs-Explicit & 1.976 & 1.987 & 1.979 & 1.980 & 2.008 & 1.986 \\
                    Bi-APNNs-Implicit & 1.994 & 1.999 & 2.003 & 2.006 & 2.007 & \textbf{2.002} \\
                    \bottomrule
                \end{tabular}
        }
            \caption{Problem I inverse problem. Comparison between three methods under different initial guesses with $\e = 10^{-8}$.}
    \label{tab:test1_ep1e-8}
\end{table}

\subsection{Problem II: Boltzmann-Poisson equation}
In the second test, we consider the Boltzmann-Poisson system \eqref{eqn:BP}. Assume the incoming boundary conditions given by 
\begin{equation}\label{BC}
f(t,x,v)\Big|_{x_L} = F_L(v), \qquad
f(t,x,-v)\Big|_{x_R} = F_R(v), \quad \text{ for } v>0, 
\end{equation}
and Maxwellian function as the initial data
$ f(x,v,t=0)= M(v) $.  
Let the applied bias voltage $V=5$, $\beta=0.002$ and doping profile $c(x)$ be given by
$$ c(x) = 1 - (1-m)\rho(0,t=0)\left[ \tanh(\frac{x-0.3}{0.02}) - \tanh(\frac{x-0.7}{0.02})\right], $$
with $m = (1-0.001)/2$.  We set $\Delta x = 0.01$ and $\Delta t = 0.005$. 

\subsubsection{Forward problem}
In this example, besides $\rho(t_i,x_i), g(t_i,x_i,v_i)$ and $f(t_i,x_i,v_i)$, which are approximated by the same neural networks in the previous example, we approximate the extra latent solution $\phi(t_i,x_i)$ by one deep neural network with $14$ hidden layers, using the empirical loss of \eqref{empirical_Loss-APNNs} and \eqref{empirical_Loss-biAPNN} for APNNs and Bi-APNNs, respectively. 

The figures \ref{fig:test2_compare_acc_rho} and \ref{fig:test2_compare_acc_phi}  display the solutions obtained from our BI-APNN framework, the standard APNN method, and the reference solutions generated by a deterministic AP solver \cite{JP2000}. We can observe that the solutions for both $\rho$ and $\phi$ from the BI-APNNs show excellent agreement with the reference solutions across all regimes, including for very small values of $\varepsilon$. While the standard APNNs also captures the overall behavior of the solution, it exhibits slight but noticeable discrepancies, which it is obviously less accurate than the implicit BI-APNNs. To quantitatively validate these observations,  we summarize the relative $\ell^2$ errors for all methods in Table~\ref{tab:test2_bi_fidelity}. The results confirm that the BI-APNNs framework consistently yields the most accurate results, with its advantage being most pronounced in the small-$\varepsilon$ regime.

\begin{figure}[H] 
    \centering
    \includegraphics[width=0.49\linewidth]{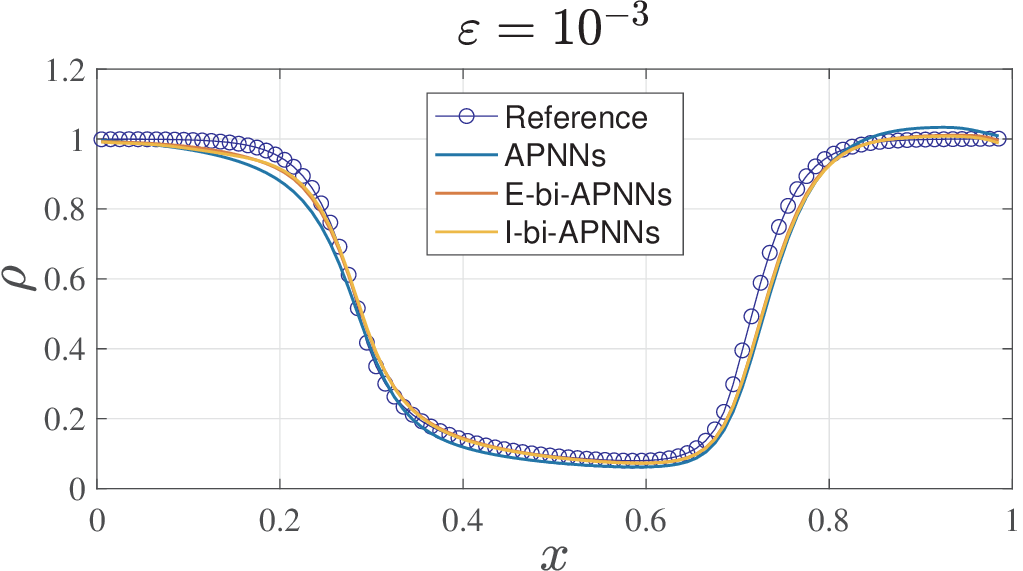}
    \includegraphics[width=0.49\linewidth]{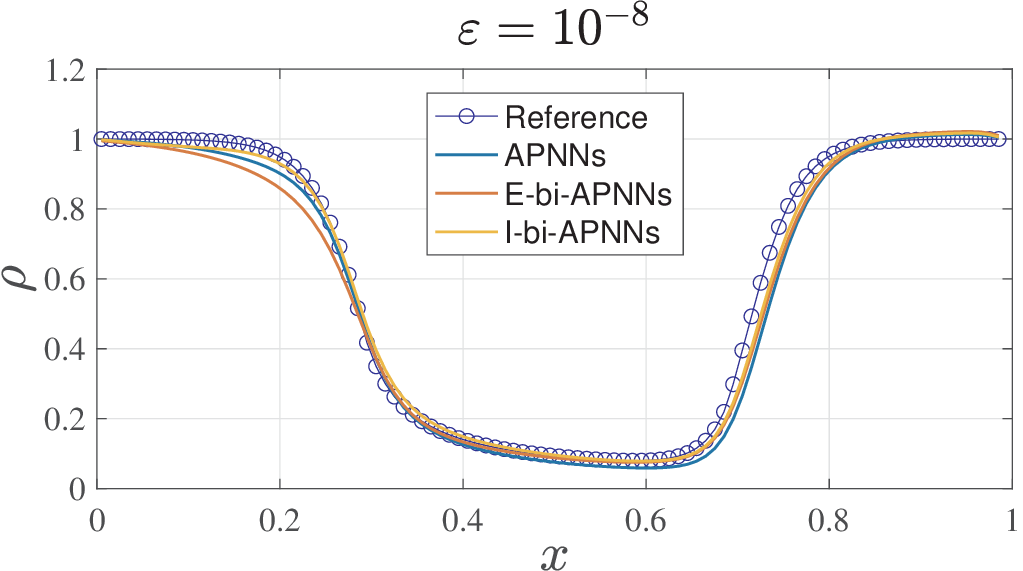}
    \caption{Bi-APNNs for deterministic problem. Comparison of $\rho$ with different methods.}
    \label{fig:test2_compare_acc_rho}
\end{figure}

\begin{figure}[H] 
    \centering
    \includegraphics[width=0.49\linewidth]{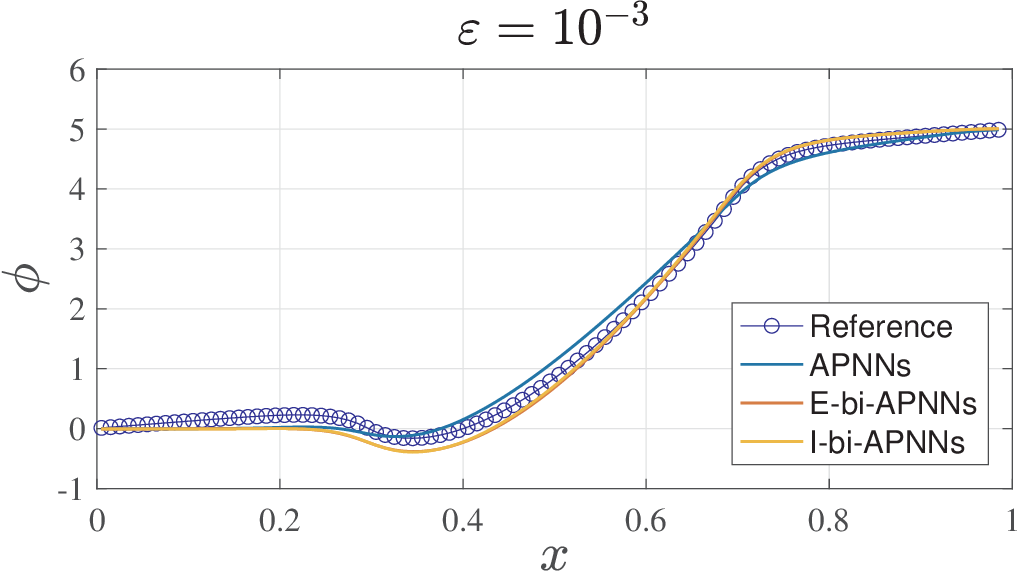}
     \includegraphics[width=0.49\linewidth]{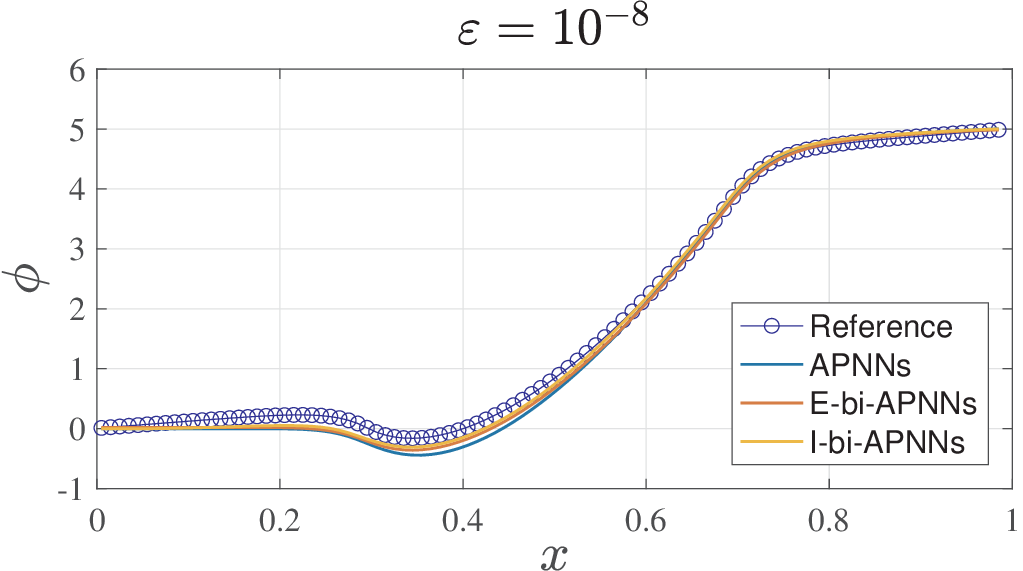}
    \caption{Bi-APNNs for deterministic problem. Comparison of $\phi$ with different methods.}
    \label{fig:test2_compare_acc_phi}
\end{figure}

\begin{table}[!htbp]
    \centering
    \begin{tabular}{l|c|c|c|c} 
        \toprule
        \multirow{2}{*}{Method} & \multicolumn{2}{c}{Relative $\ell^2$ error ($\rho$)} & \multicolumn{2}{c}{Relative $\ell^2$ error ($\phi$)} \\
        \cmidrule(lr){2-3} \cmidrule(lr){4-5}
        & $\e=10^{-3}$ & $\e=10^{-8}$ & $\e=10^{-3}$ & $\e=10^{-8}$ \\
        \midrule
        APNNs & $6.28 \times 10^{-2}$ &$6.23\times10^{-2}$ & $5.78 \times 10^{-2}$ & $6.16 \times 10^{-2}$ \\
        Bi-APNNs-Explicit & $\mathbf{4.29 \times 10^{-2}}$ & $6.22\times10^{-2}$ & $5.71 \times 10^{-2}$ & $4.85 \times 10^{-2}$\\
        Bi-APNNs-Implicit & $4.46 \times 10^{-2}$ & $\mathbf{3.83 \times 10^{-2}}$ & $\mathbf{4.42 \times 10^{-2}}$ &  $ \mathbf{3.90 \times 10^{-2}}$\\
        \bottomrule
    \end{tabular}
    \caption{Problems II. Relative $\ell^2$ error comparison between reference solution and other methods with different $\epsilon$ at the final time at $T=0.1$. }
    \label{tab:test2_bi_fidelity}
\end{table}

To further demonstrate that the BI-APNNs framework also offers an advantage in terms of training efficiency, particularly in the small-$\varepsilon$ regime, we plot the training loss in Figure \ref{fig:test2BI_AP_training}. The plots clearly show that for small values of $\varepsilon$ (e.g., $10^{-3}$ and $10^{-8}$), the loss for the BI-APNNs converges  faster than that of the standard APNNs.  This confirms that the bi-fidelity formulation not only improves accuracy but also accelerates the training process in the challenging fluid-dynamic limit.
\begin{figure}[H]
    \centering
    \includegraphics[width=0.49\linewidth]{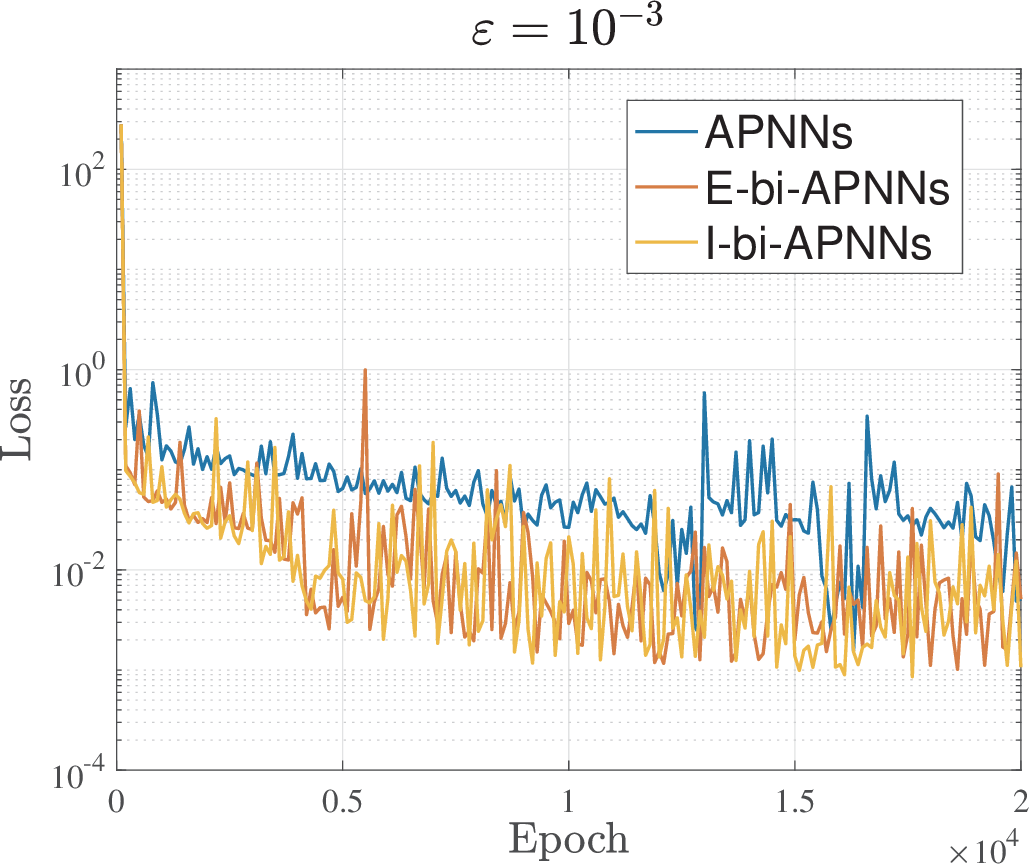}
    \includegraphics[width=0.49\linewidth]{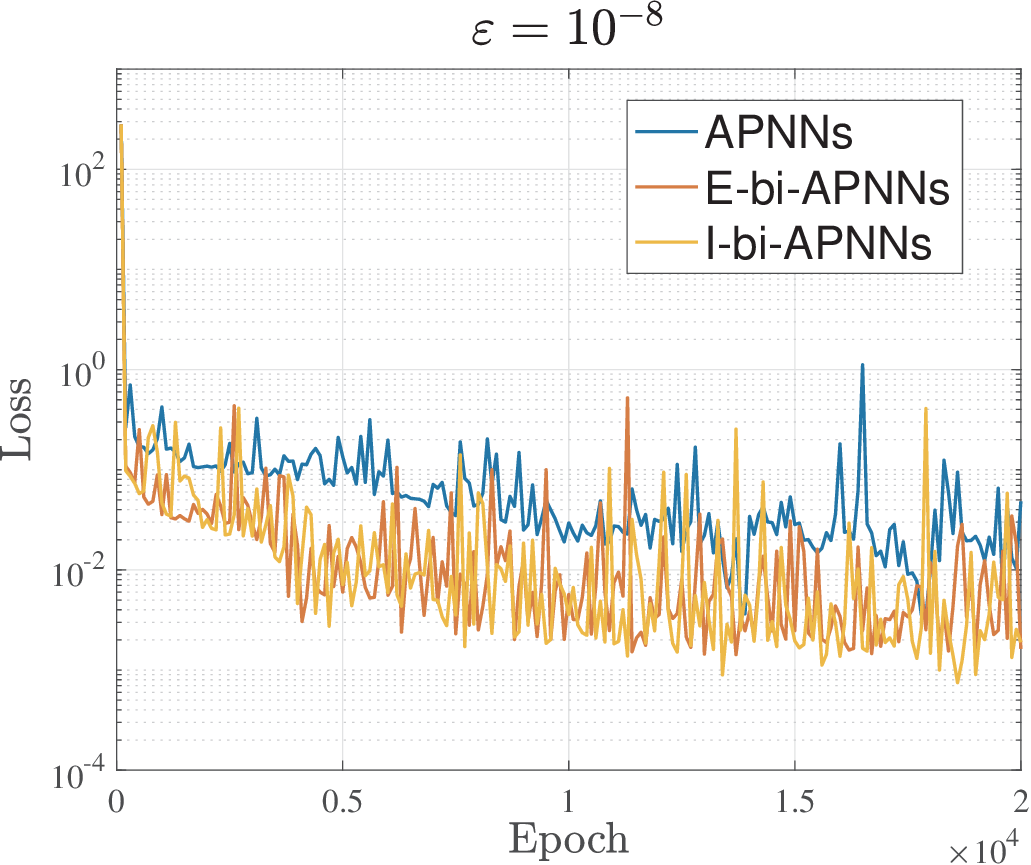}
    \caption{Problem II with $\e=10^{-3}$ and $\e = 10^{-8}$. Comparison of training losses between Bi-APNNs and APNNs.}
    \label{fig:test2BI_AP_training}
\end{figure}

The BI-APNNs framework, particularly in the small-$\varepsilon$ regime, significantly enhances the convergence speed. This is demonstrated in Figure~\ref{fig:BI_AP_relative_test2}, which compares the convergence behavior of the relative $\ell^2$ error for both the BI-APNNs and standard APNNs under identical network settings.The plots clearly show that the losses for the BI-APNNs converges substantially faster than that of the APNNs. This accelerated convergence in the fluid-dynamic limit is a key practical benefit of the proposed bi-fidelity approach.

\begin{figure}[H]
    \centering
    \includegraphics[width=0.49\linewidth]{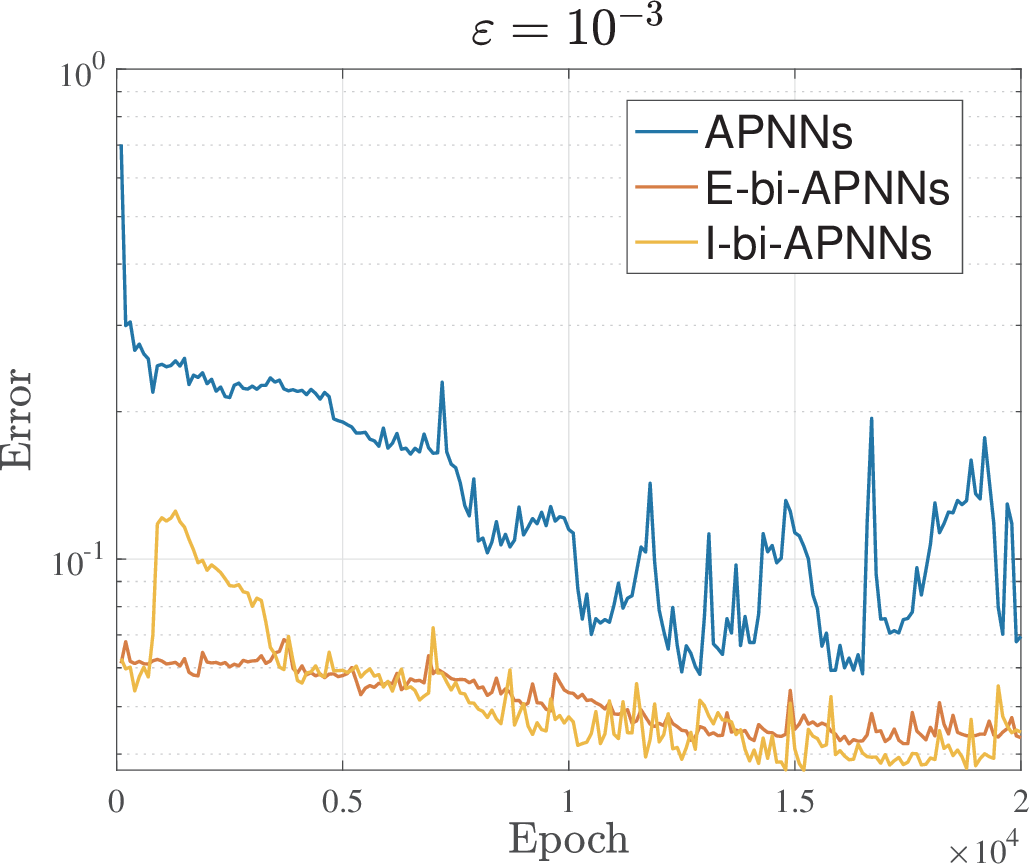}
    \includegraphics[width=0.49\linewidth]{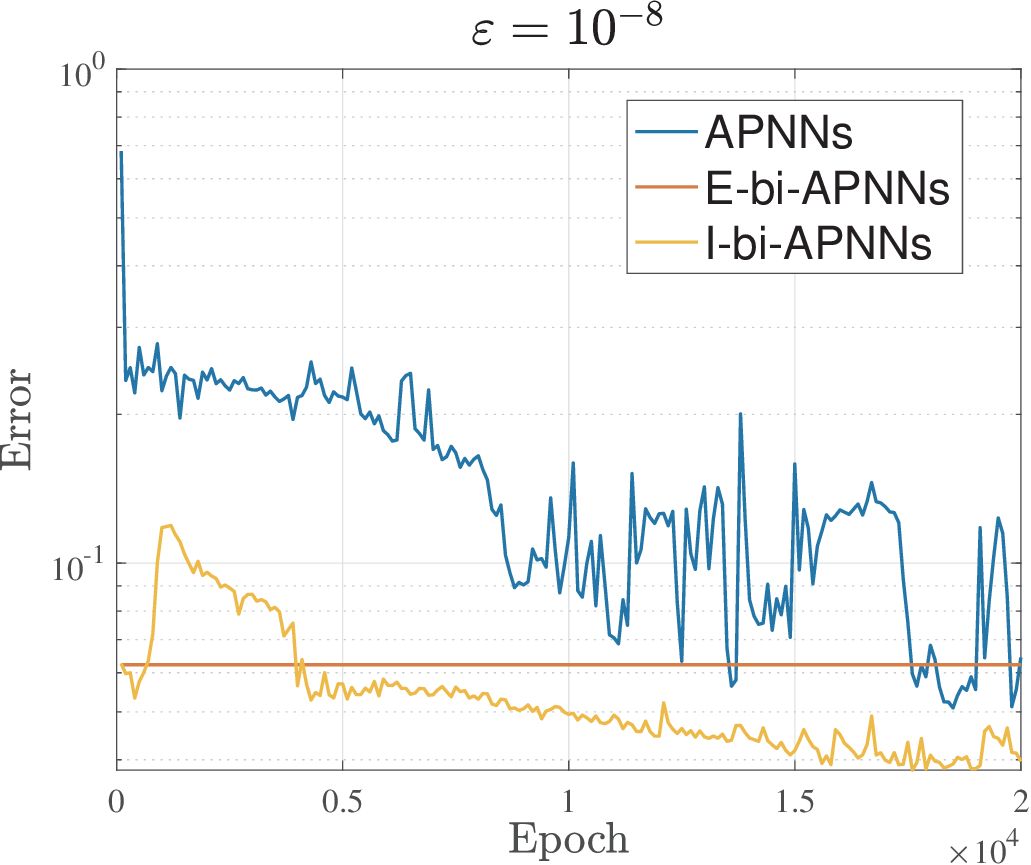}
    \caption{Problem II with $\e=10^{-3}$ and $\e = 10^{-8}$. Comparison of convergence in relative errors of Bi-APNNs and APNNs.}
    \label{fig:BI_AP_relative_test2}
\end{figure}

We further investigate the impact of network architecture on the performance of both the APNNs and BI-APNNs frameworks. The empirical loss for the pretrained diffusion system can be found in \eqref{Loss-diffusion_empirical2}. Based on our findings in Problem I, where the number of neurons was observed to have a marginal impact on performance, we focus this study on varying the number of hidden layers for the network approximating the density, $\rho$. The results, summarized in Tables~\ref{tab:test2_all_methods_layer_1e-3} and \ref{tab:test2_all_methods_layer_1e-8}, reveal the contrast in architectural requirements. The standard APNNs require a relatively deep network, typically with 3 to 4 hidden layers, to achieve a relative error on the order of $\mathcal{O}(10^{-2})$. In sharp contrast, for small Knudsen numbers, the BI-APNNs can achieve a comparable or even higher accuracy without any hidden layers at all for the correction term $\rho_{\text{corr}}$. This remarkable result underscores the efficiency of our bi-fidelity formulation, demonstrating its ability to significantly reduce computational resource requirements while maintaining high accuracy.

\begin{table}[!htbp]
    \centering
    \scalebox{0.89}{
    \begin{tabular}{l|c|c|c|c|c}
        \toprule
        \multirow{2}{*}{Method}  & \multicolumn{5}{c}{Layer} \\
        \cmidrule(lr){2-6}
        & 0 & 1 & 2 & 3 & 4   \\
        \midrule
        \multirow{1}{*}{APNNs} & $2.94 \times 10^{-1}$ & $1.35 \times 10^{-1}$ & $6.60 \times 10^{-2}$ & $3.85 \times 10^{-2}$ & $4.31 \times 10^{-2}$ \\
        \midrule
        \multirow{1}{*}{Bi-APNNs-Explicit} & $5.20 \times 10^{-2}$ & $4.13 \times 10^{-2}$ & $5.66 \times 10^{-2}$ & $5.73 \times 10^{-2}$ & $4.10 \times 10^{-2}$ \\
        \midrule
        \multirow{1}{*}{Bi-APNNs-Implicit}  & $5.02 \times 10^{-2}$ & $4.38 \times 10^{-2}$ & $3.99 \times 10^{-2}$ & $4.19 \times 10^{-2}$ & $3.60 \times 10^{-2}$ \\
        \bottomrule
    \end{tabular}
    } 
    \caption{Problems II. Comparison of Relative Error for different methods across various network layers, with $\varepsilon=10^{-3}$.}
    \label{tab:test2_all_methods_layer_1e-3}
\end{table}

\begin{table}[!htbp]
    \centering
    \scalebox{0.89}{
    \begin{tabular}{l|c|c|c|c|c}
        \toprule
        \multirow{2}{*}{Method} &  \multicolumn{5}{c}{Layer} \\
        \cmidrule(lr){2-6}
        & 0 & 1 & 2 & 3 & 4   \\
        \midrule
        \multirow{1}{*}{APNNs} &  $2.63 \times 10^{-1}$ & $8.20 \times 10^{-2}$ & $4.11 \times 10^{-2}$ & $4.26 \times 10^{-2}$ & $4.50 \times 10^{-2}$ \\
        \midrule
        \multirow{1}{*}{Bi-APNNs-Explicit}  & $6.22 \times 10^{-2}$ & $6.22 \times 10^{-2}$ & $6.22 \times 10^{-2}$ & $6.22 \times 10^{-2}$ & $6.22 \times 10^{-2}$ \\
        \midrule
        \multirow{1}{*}{Bi-APNNs-Implicit}  & $5.07 \times 10^{-2}$ & $4.68 \times 10^{-2}$ & $3.76 \times 10^{-2}$ & $3.85 \times 10^{-2}$ & $3.61 \times 10^{-2}$ \\
        \bottomrule
    \end{tabular}
    } 
    \caption{Problems II. Comparison of relative Error for different methods across various network layers, with $\varepsilon=10^{-8}$.}
    \label{tab:test2_all_methods_layer_1e-8}
\end{table}

\subsubsection{Inverse problem}
Next, we further examine the performance for a inverse problem inferring the scattering coefficient from available measurement data in BP system using the Bi-APNNs and APNNs formulation. Similarly, we consider using both methods when partial observation data is available. We obtain synthetic dataset for $\rho(t_d^i,x_d^i)$ and $\phi(t_d^i,x_d^i)$, where $\rho(t_d^i,x_d^i)$ is obtained in the way similar to Problem I. Besides the APNNs loss \eqref{empirical_Loss-APNNs} and Bi-APNNs loss \eqref{empirical_Loss-biAPNN}, we need an extra term $\mathcal{L}_{\text{data}}^{\phi}(\theta)$, 

\begin{equation}
    \mathcal{L}_{\text{data}}^{\phi}(\theta) = \frac{w_d^{\phi}}{N_{d_2}} \sum_{i=1}^{N_{d_2}} \left| \phi^{\mathrm{NN}}(t_{d}^i, x_{d}^i; \theta) -\phi(t_{d}^i,x_{d}^i) \right|^2,
    \label{Inverse_phi}
\end{equation}
where $\omega_d^\phi = 1$, and $N_{d_2}$ is the number of data points for $\phi$. We train the network model on measurements with $N_{d_1}, N_{d_2} = 100$ samples randomly selected in the domain $(t,x) \in [0,0.1] \times [0,1]$ to estimate the target parameter $\sigma=2$ under different initial guesses $\sigma_0 = 0.5, 1.0, 1.5, 1.7. 1.9$ under both $\varepsilon = 10^{-3}$ and $\varepsilon = 10^{-8}$. 

The enhanced performance of the BI-APNNs for the inverse problem is demonstrated both visually and quantitatively in Figure~\ref{fig:test2_inverse} and Tables~\ref{tab:test2_ep1e-3} and \ref{tab:test2_ep1e-8}, respectively. The figure illustrates a more accurate and stable convergence for the BI-APNNs, showing the estimated mean of the parameter $\hat{\sigma}$ with a tighter confidence interval compared to the standard APNNs. This qualitative observation is confirmed by the quantitative results in the tables, which show that our two BI-APNNs formulations achieve a better inverse parameter estimation. 

\begin{figure}[H]
    \centering
    \includegraphics[width=0.49\linewidth]{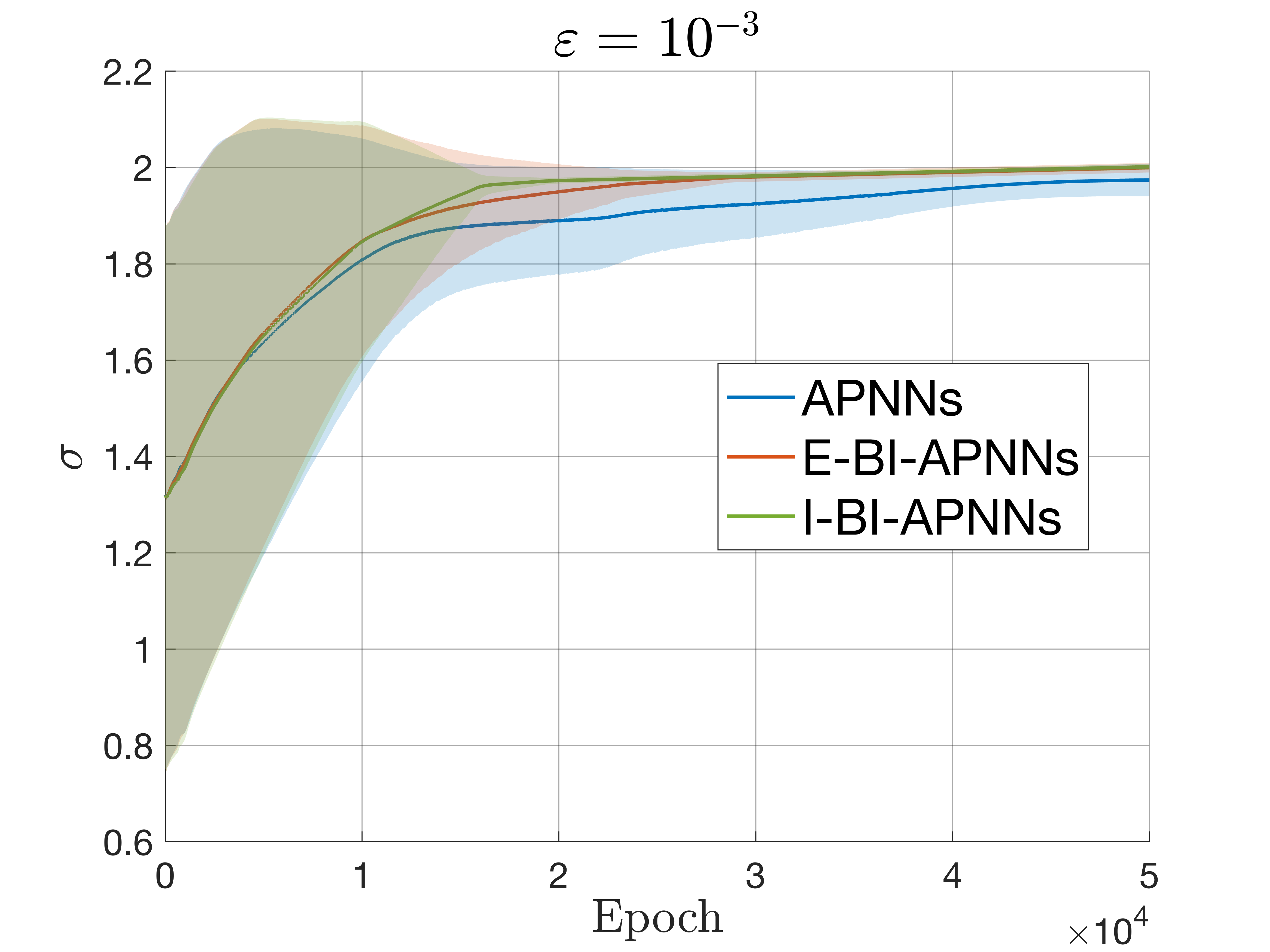}
    \includegraphics[width=0.49\linewidth]{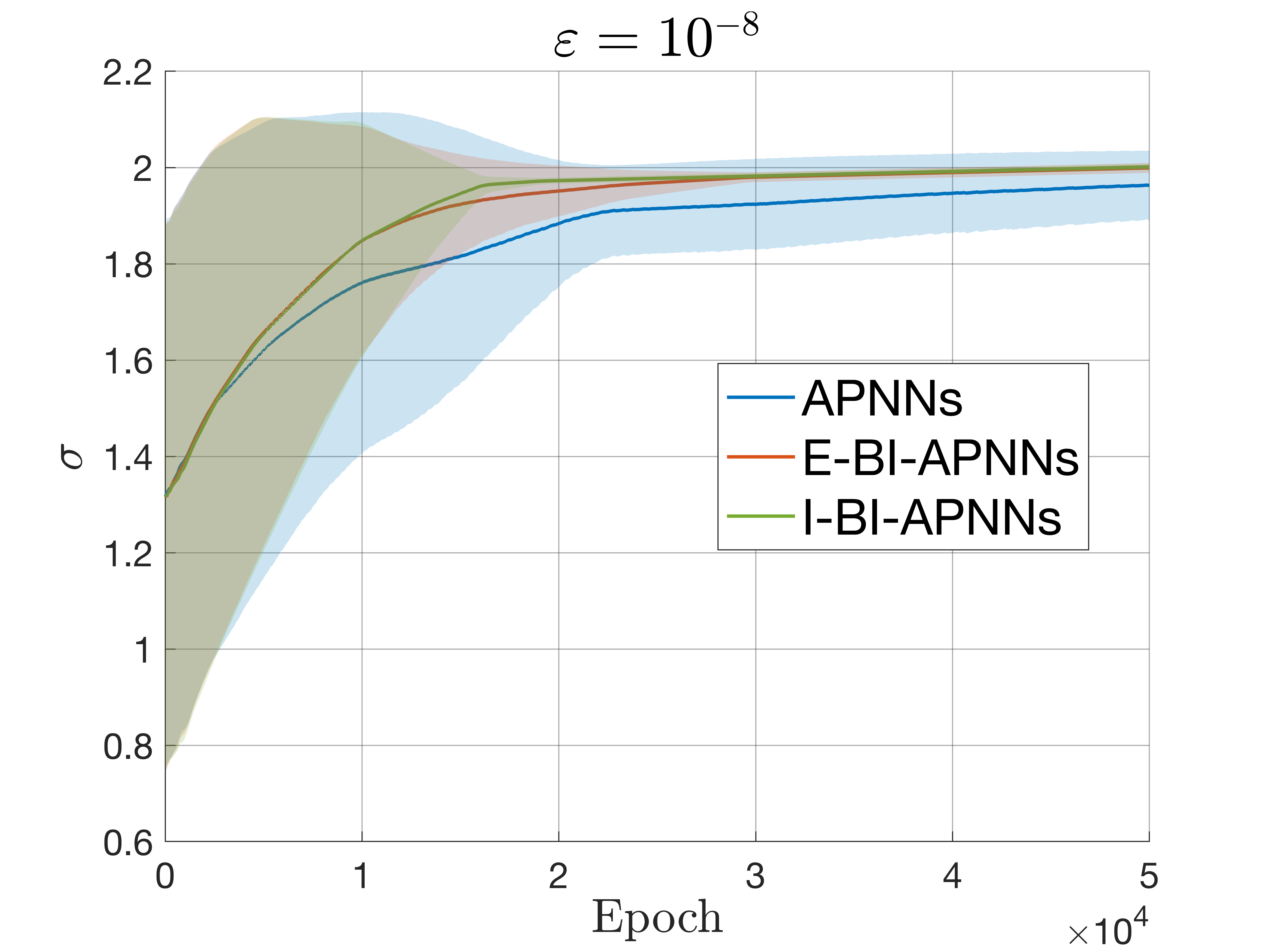}
    \caption{Inverse problem for Problem II. Comparison between three methods under different initial guesses with $\e = 10^{-3}$ and $\e = 10^{-8}$.}
    \label{fig:test2_inverse}
\end{figure}

\begin{table}[H]
    \centering
    \vspace{-15pt}
            \renewcommand{\multirowsetup}{\centering}
            \setlength{\tabcolsep}{5.5pt}
            \scalebox{1}{
                \begin{tabular}{l|c|c|c|c|c|c}
                    \toprule
                    \multirow{2}{*}{Method} & \multicolumn{6}{c}{Initial Guess ($\sigma_0$)} \\
                    \cmidrule(lr){2-7}
                    & 0.5 & 1.0 & 1.5 & 1.7 & 1.9 & Average \\
                    \midrule
                    APNNs & 1.966 & 1.964 & 1.927 & 2.006 & 2.009 & 1.974 \\
                    Bi-APNNs-Explicit & 1.983 & 1.998 & 2.004 & 2.006 & 2.008 & \textbf{2.000} \\
                    Bi-APNNs-Implicit & 1.992 & 1.999 & 2.004 & 2.006 & 2.007 & 2.002 \\
                    \bottomrule
                \end{tabular}
        }
            \caption{Problem II inverse problem. Comparison between three methods under different initial guesses with $\e = 10^{-3}$.}
    \label{tab:test2_ep1e-3}
\end{table}

\begin{table}[H]
    \centering
    \vspace{-15pt}
            \renewcommand{\multirowsetup}{\centering}
            \setlength{\tabcolsep}{5.5pt}
            \scalebox{1}{
                \begin{tabular}{l|c|c|c|c|c|c}
                    \toprule
                    \multirow{2}{*}{Method} & \multicolumn{6}{c}{Initial Guess ($\sigma_0$)} \\
                    \cmidrule(lr){2-7}
                    & 0.5 & 1.0 & 1.5 & 1.7 & 1.9 & Average \\
                    \midrule
                    APNNs & 1.838 & 1.962 & 2.002 & 2.005 & 2.006 & 1.963 \\
                    Bi-APNNs-Explicit & 1.982 & 1.997 & 2.004 & 2.006 & 2.007 & \textbf{1.999} \\
                    Bi-APNNs-Implicit & 1.992 & 1.998 & 2.004 & 2.006 & 2.007 & 2.002 \\
                    \bottomrule
                \end{tabular}
        }
            \caption{Problem II inverse problem. Comparison between three methods under different initial guesses with $\e = 10^{-8}$.}
    \label{tab:test2_ep1e-8}
\end{table}

%% file: Appendix.tex
\subsection{Velocity discretizations}
For completeness, we briefly mention the velocity discretization similar to what has been studied in \cite{JP00}. Set $f(t,x,v) = \psi(t,x,v) M(v)$, where $M(v) = \frac{1}{\sqrt{\pi}} e^{-v^2}$, with 
\begin{equation}\label{Psi} \psi(t,x,v) = \sum_{k=0}^N \psi_k(t,x) \tilde H_k(v), \end{equation}
being the Hermite expansion. 
For notation simplicity, we omit the $t$ and $x$ dependence of functions below. 
Here $\tilde H_k$ are the renormalized Hermite polynomials
defined as $\tilde H_{-1}=0$, $\tilde H_0 = 1/\pi^{1/4}$ and 
$$ \tilde H_{j+1} = v \sqrt{\frac{2}{j+1}}\tilde H_j - \sqrt{\frac{j}{j+1}}\tilde H_{j-1} \quad \text{for } j\geq 0, $$
satisfying $\partial_v \tilde H_j = \sqrt{2 j}\, \tilde H_{j-1}$. 
The inverse Hermite expansion is given by 
\begin{equation}\label{I-Psi} \psi_k = \sum_{j=0}^N \psi(v_j)\, \tilde H_k(v_j)\, w_j, 
\end{equation}
where $(v_j, w_j)$ are the points and corresponding weights of the Gauss-Hermite quadrature rule. 
Thus the collision operator $Q$ in \eqref{Boltz-eqn} can be computed by
$$ Q(f)(v) = M(v) \sum_{j=0}^N \sigma(v, v_j)\, \psi(v_j)\, w_j - \lambda(v) f(v), $$
with $\lambda(v) = \sum_{j=0}^N \sigma(v, v_j)\, w_j$. From \eqref{Psi} and \eqref{I-Psi}, one computes the derivative in $v$ by
\begin{equation*}
\begin{aligned}
\partial_v \psi & = \sum_{k=0}^N \psi_k\, \partial_v \tilde H_k(v) 
= \sum_{k=0}^N \psi_k \sqrt{2k}\, \tilde H_{k-1}(v) \\
& = \sum_{k=0}^N \sum_{j=0}^N \psi(v_j) \, \tilde H_k(v_j) w_j \sqrt{2k}\, \tilde H_{k-1}(v) \\
& = \sum_{j=0}^N \psi(v_j)\, C_j(v),
\end{aligned}
\end{equation*}
where $C_j(v) = \sum_{k=0}^N \sqrt{2k}\, \tilde H_k(v_j) \tilde H_{k-1}(v) w_j$, 
and can be precomputed before the time iteration.  

Since $f = \psi M$, instead of solving the model equation \eqref{Boltz-eqn} for $f$, we transform it into solving an equation for the function $\psi$: 
$$ \e \partial_t \psi + v \partial_x\psi  + \partial_x\phi \left(\partial_v \psi - 2 v \psi \right) = \frac{1}{\e}\, \tilde Q(\psi), $$
where $$\tilde Q(\psi)(v) = \sum_{j=0}^N \sigma(v, v_j)\, \psi(v_j)\, w_j - \lambda(v)\psi(v). $$

\subsection{Loss functions for PINNs}
The empirical risk for PINNs is as follows:
\small{\begin{equation}
\label{empirical_Loss-PINN}
\begin{aligned}
\mathcal{R}_{\mathrm{PINNs}}^{\e} & = 
 \frac{1}{N_1} \sum_{i=1}^{N_1} \Big| \e \partial_t f_\theta^{\mathrm{NN}}(t_i,x_i,v_i) +  \bv \partial_x f_\theta^{\mathrm{NN}}(t_i,x_i,v_i)
+ \partial_x \phi(t_i,x_i)\, \partial_v f_\theta^{\mathrm{NN}}(t_i,x_i,v_i) - \frac{1}{\e}\mathcal{Q}(f_\theta^{\mathrm{NN}}(t_i,x_i,v_i))|^2 \,\\[6pt]
& + \frac{\lambda_1}{N_2} \sum_{i=1}^{N_2} \left|\mathcal{B} (f_{\theta}^{\mathrm{NN}}(t_i,x_i,v_i)) - f_{\text{BC}}(t_i,x_i,v_i)\right|^2  + \frac{\lambda_2}{N_3} \sum_{i=1}^{N_3} \left | 
\mathcal{I}(f_\theta^{\mathrm{NN}}(t_i,x_i,v_i))- f_{\text{IC}}(t_i,x_i,v_i)\right|^2,
\end{aligned}
\end{equation}
}
where $N_1$, $N_2$, $N_3$ are the number of sample points of $\mathcal{T} \times \mathcal{D} \times \Omega$, $\mathcal{T} \times \partial \mathcal{D} \times \Omega$ and $\mathcal{D} \times \Omega$. $\lambda_1$ and $\lambda_2$ are the corresponding weights. 

\subsection{Loss functions for APNNs}
The empirical risk for APNNs is as follows:
\eqref{Macro}--\eqref{Micro} as the APNNs loss function: 
\small{\begin{equation}
\label{empirical_Loss-APNNs}
\begin{aligned}
\mathcal{R}_{\mathrm{APNNs}}^{\e} = & \frac{1}{N_1} \sum_{i=1}^{N_1}\left|\partial_t \rho_\theta^{\mathrm{NN}}(t_i,x_i)+\nabla_{\bx} \cdot\left\langle\bv g_\theta^{\mathrm{NN}}(t_i,x_i)\right\rangle + \red{ \nabla_{\bx}\phi(t_i,x_i) \cdot  \left\langle \nabla_{\bv} g_\theta^{\mathrm{NN}}(t_i,x_i)\right\rangle } \right|^2 \,\\[6pt]
& + \frac{1}{N_2} \sum_{i=1}^{N_2} \Big| \e^2 \partial_t g_\theta^{\mathrm{NN}}(t_i,x_i,v_i)
 + \e(I-\Pi) \left(\bv \cdot \nabla_{\bx} g_\theta^{\mathrm{NN}}(t_i,x_i,v_i) + \red{\nabla_{\bx}\phi(t_i,x_i) \cdot \nabla_{\bv} g_\theta^{\mathrm{NN}}(t_i,x_i,v_i) }\right) \\[6pt]
&   - 2\bv \cdot \nabla_{\bx}\phi(t_i,x_i)\, \rho_\theta^{\mathrm{NN}}(t_i,x_i)\, M(v)  
 + \bv \cdot \nabla_{\bx} \rho_\theta^{\mathrm{NN}}(t_i,x_i) M(v)  - \mathcal{Q}( g_\theta^{\mathrm{NN}}(t_i,x_i,v_i) ) \Big|^2 \,\\[6pt]
& + \frac{\lambda_1}{N_3} \left|\mathcal{B}  \left(\rho_\theta^{\mathrm{NN}}(t_i,x_i) M(v) + \e g_\theta^{\mathrm{NN}}(t_i,x_i,v_i)\right) - f_{\text{BC}}(t_i,x_i,v_i)\right|^2 \,\\[6pt]
 &+\frac{\lambda_2}{N_4} \left|\mathcal{I}\left(\rho_\theta^{\mathrm{NN}}(t_i,x_i) M(v) +\e g_\theta^{\mathrm{NN}}(t_i,x_i,v_i)\right)- f_{\text{IC}}(t_i,x_i,v_i) \right|^2, 
\end{aligned}
\end{equation}
}
where $N_1$,$N_2$, $N_3$, $N_4$ are the number of sample points of $\mathcal{T} \times \mathcal{D}$, $\mathcal{T} \times \mathcal{D} \times \Omega$, $\mathcal{T} \times \partial \mathcal{D} \times \Omega$ and $\mathcal{D} \times \Omega$. $\lambda_1$ and $\lambda_2$ are the corresponding weights.

Regarding the incoming boundary condition given as
$$ f(t,x,v)\Big|_{x_L} = F_L(v), \qquad
f(t,x,-v)\Big|_{x_R} = F_R(v), \quad \text{ for } v>0, 
$$
we look at the third term of \eqref{Loss-APNN}, with the discretized form of the integral shown by
\begin{equation*}
\begin{aligned}
&\sum_{i}\sum_{j \text{ for }v_j>0} 
\left| \rho_\theta^{\mathrm{NN}}(t_i, x_L) M(v_j) + 
\e g_\theta^{\mathrm{NN}}(t_i, x_L, v_j) - F_L(v_j) \right|^2 w_j \, \Delta t \\
& + \sum_{i}\sum_{j \text{ for }v_j<0}
\left| \rho_\theta^{\mathrm{NN}}(t_i, x_R) M(v_j) + 
\e g_\theta^{\mathrm{NN}}(t_i, x_R, v_j) - F_R(v_j) \right|^2 w_j\, \Delta t. 
\end{aligned}
\end{equation*}

\subsection{Loss functions for diffusion system} 
The empirical risk for the diffusion system \eqref{Diffusion} is as follows:
\begin{equation}
\label{Loss-diffusion_empirical}
\begin{aligned}
\mathcal{R}_{\text{diffusion}_1} 
&= \frac{1}{N_1} \sum_{i=1}^{N_1} \left| \partial_t \rho_{\textit{diff},\theta}^{\mathrm{NN}}(t_i, x_i) - \nabla_x \cdot \left( T \nabla_x \rho_{\textit{diff},\theta}^{\mathrm{NN}}(t_i, x_i) - 2 \rho_{\textit{diff},\theta}^{\mathrm{NN}}(t_i, x_i) \nabla_x \phi(t_i, x_i) \right) \right|^2 \\[6pt]
&+ \frac{\lambda_1^{\text{diff}}}{N_2} \sum_{i=1}^{N_2} \left| \mathcal{B} \left( \rho_{\textit{diff},\theta}^{\mathrm{NN}}(t_i, x_i)M(v) \right) - f_{\text{BC}}(t_i, x_i,v_i) \right|^2 + \frac{\lambda_2^{\text{diff}}}{N_3} \sum_{i=1}^{N_3} \left| \mathcal{I}\left( \rho_{\textit{diff},\theta}^{\mathrm{NN}}(x_i)M(v) \right) - f_{\text{IC}}(x_i,v_i) \right|^2,
\end{aligned}
\end{equation}
where $N_1$, $N_2$, $N_3$ are the number of sample points of $\mathcal{T} \times \mathcal{D} \times \Omega$, $\mathcal{T} \times \partial \mathcal{D} \times \Omega$ and $\mathcal{D} \times \Omega$. $\lambda_1$ and $\lambda_2$ are the corresponding weights. 

Another optional diffusion loss comes from \eqref{limit-diff}. The corresponding empirical risk loss is as follows:
\begin{equation}
\label{Loss-diffusion_empirical2}
\begin{aligned}
\mathcal{R}_{\text{diffusion}_2} 
&= \frac{1}{N_1} \sum_{i=1}^{N_1} \left| \partial_t \rho_{\textit{diff},\theta}^{\mathrm{NN}}(t_i, x_i) + \nabla_x \cdot \left\langle v g_{\theta}^{\mathrm{NN}}(t_i, x_i, v) \right\rangle + \nabla_x \phi(t_i, x_i) \cdot \left\langle \nabla_v g_{\theta}^{\mathrm{NN}}(t_i, x_i, v) \right\rangle \right|^2 \\[6pt]
&+ \frac{1}{N_2} \sum_{i=1}^{N_2} \left| \bv \cdot \nabla_x \rho_{\textit{diff},\theta}^{\mathrm{NN}}(t_i, x_i) \, M(v) - 2 \bv \cdot \nabla_x \phi(t_i, x_i) \, \rho_{\textit{diff},\theta}^{\mathrm{NN}}(t_i, x_i) \, M(v) - \mathcal{Q}\left( g_{\theta}^{\mathrm{NN}}(t_i, x_i, v_i) \right) \right|^2 \\[6pt]
&+ \frac{\lambda_1^{\text{diff}}}{N_3} \sum_{i=1}^{N_3} \left| \mathcal{B} \left( \rho_{\textit{diff},\theta}^{\mathrm{NN}}(t_i, x_i) M(v_i) \right) - f_{\text{BC}}(t_i, x_i, v_i) \right|^2 + \frac{\lambda_2^{\text{diff}}}{N_4} \sum_{i=1}^{N_4} \left| \mathcal{I}\left( \rho_{\textit{diff},\theta}^{\mathrm{NN}}(x_i) M(v) \right) - f_{\text{IC}}(x_i, v_i) \right|^2,
\end{aligned}
\end{equation}
where $N_1$, $N_2$ are the number of sample points of $\mathcal{T} \times \mathcal{D} \times \Omega$, $N_3$  and $N_4$ are the number of sample points of  $\mathcal{T} \times \partial \mathcal{D} \times \Omega$ and $\mathcal{D} \times \Omega$. $\lambda_1$ and $\lambda_2$ are the corresponding weights. 


\subsection{Loss functions for BI-APNNs} 
The empirical risk for BI-APNNs is as follows:
\begin{equation}
\label{empirical_Loss-biAPNN}
\begin{aligned}
\mathcal{R}_{\mathrm{BI-APNNs}}^{\e} = & \frac{1}{N_1} \sum_{i=1}^{N_1}\left|\partial_t \rho_{\textit{bi},\theta}^{\mathrm{NN}}(t_i,x_i)+\nabla_{\bx} \cdot\left\langle\bv g_\theta^{\mathrm{NN}}(t_i,x_i)\right\rangle + \red{ \nabla_{\bx}\phi(t_i,x_i) \cdot  \left\langle \nabla_{\bv} g_\theta^{\mathrm{NN}}(t_i,x_i)\right\rangle } \right|^2 \,\\[6pt]
& + \frac{1}{N_2} \sum_{i=1}^{N_2} \Big| \e^2 \partial_t g_\theta^{\mathrm{NN}}(t_i,x_i,v_i)
 + \e(I-\Pi) \left(\bv \cdot \nabla_{\bx} g_\theta^{\mathrm{NN}}(t_i,x_i,v_i) + \red{\nabla_{\bx}\phi(t_i,x_i) \cdot \nabla_{\bv} g_\theta^{\mathrm{NN}}(t_i,x_i,v_i) }\right) \\[6pt]
&   - 2\bv \cdot \nabla_{\bx}\phi(t_i,x_i)\, \rho_{\textit{bi},\theta}^{\mathrm{NN}}(t_i,x_i)\, M(v)  
 + \bv \cdot \nabla_{\bx} \rho_{\textit{bi},\theta}^{\mathrm{NN}}(t_i,x_i) M(v)  - \mathcal{Q}( g_\theta^{\mathrm{NN}}(t_i,x_i,v_i) ) \Big|^2 \,\\[6pt]
& + \frac{\lambda_1}{N_3} \left|\mathcal{B}  \left(\rho_{\textit{bi},\theta}^{\mathrm{NN}}(t_i,x_i) M(v) + \e g_\theta^{\mathrm{NN}}(t_i,x_i,v_i)\right) - f_{\text{BC}}(t_i,x_i,v_i)\right|^2 \,\\[6pt]
 &+\frac{\lambda_2}{N_4} \left|\mathcal{I}\left(\rho_{\textit{bi},\theta}^{\mathrm{NN}}(t_i,x_i) M(v) +\e g_\theta^{\mathrm{NN}}(t_i,x_i,v_i)\right)- f_{\text{IC}}(t_i,x_i,v_i) \right|^2
 \,.
\end{aligned}
\end{equation}
Where $N_1$,$N_2$, $N_3$, $N_4$ are the number of sample points of $\mathcal{T} \times \mathcal{D}$, $\mathcal{T} \times \mathcal{D} \times \Omega$, $\mathcal{T} \times \partial \mathcal{D} \times \Omega$ and $\mathcal{D} \times \Omega$. $\lambda_1$ and $\lambda_2$ are the corresponding weights.

For the integrals in losses of PINNs, APNNs and Bi-APNNs, the operator $\langle \cdot \rangle$ in  \eqref{bracket} and $\Pi(\cdot)$ in  \eqref{Ans}, we compute them by quadrature rule.